\documentclass[a4paper,oneside,10pt]{article}%
\usepackage[english]{babel}
\usepackage[T1]{fontenc}
\usepackage[a4paper,top=3cm,bottom=3cm,left=3cm,right=3cm,marginparwidth=1.75cm]%
{geometry}
\usepackage[colorlinks=true, allcolors=blue, hypertexnames=false]{hyperref}
\usepackage{empheq}
\usepackage{graphicx}
\usepackage[colorinlistoftodos]{todonotes}
\usepackage{amsmath}
\usepackage{mathtools}
\usepackage{upgreek}
\usepackage{amssymb}
\usepackage{amsfonts}
\usepackage{amsthm}
\usepackage{hyperref}
\usepackage{enumitem}
\usepackage{mathrsfs}
\usepackage{fontenc}
\usepackage{verbatim}
\usepackage{framed}
\usepackage{algorithm}
\usepackage{bbm}
\usepackage{algpseudocode}
\usepackage{appendix}
\usepackage{color}
\usepackage{multirow}%
\providecommand{\U}[1]{\protect\rule{.1in}{.1in}}

\newtheorem{theorem}{Theorem}[section]

\newtheorem{corollary}[theorem]{Corollary}

\newtheorem{assumption}[theorem]{Assumption}

\newtheorem{lemma}[theorem]{Lemma}

\newtheorem{proposition}[theorem]{Proposition}
\newtheorem{remark}[theorem]{Remark}

\numberwithin{equation}{section}

\begin{document}

\title{A Global Stochastic Maximum Principle for Forward-Backward Stochastic
Control Systems with Quadratic Generators}
\author{Mingshang Hu\thanks{Zhongtai Securities Institute for Financial Studies, Shandong University, Jinan, Shandong
250100, PR China. humingshang@sdu.edu.cn. Research supported by National Natural Science
Foundation of China (No. 11671231).}
\and Shaolin Ji\thanks{Zhongtai Securities Institute for Financial Studies,
Shandong University, Jinan, Shandong 250100, PR China. Email: jsl@sdu.edu.cn
(Corresponding author). Research supported by the National Natural Science
Foundation of China (No. 11971263; 11871458).}
\and Rundong Xu\thanks{Zhongtai Securities Institute for Financial Studies,
Shandong University, Jinan, Shandong 250100, PR China. Email:
rundong.xu@mail.sdu.edu.cn.}}
\maketitle

\textbf{Abstract}. We study a stochastic optimal control problem for
forward-backward control systems with quadratic generators. In order to
establish the first and second-order variational and adjoint equations, we
obtain a new estimate for one-dimensional linear BSDEs with unbounded
stochastic Lipschitz coefficients involving bounded mean oscillation
martingales (BMO-martingales for short) and prove the solvability for a class
of multi-dimensional BSDEs with this type. Finally, a new global stochastic
maximum principle is deduced.

{\textbf{Key words}. BMO-martingale; Forward-backward stochastic control
systems; Linear BSDEs with unbounded coefficients; Maximum principle; Quadratic BSDEs }

\textbf{AMS subject classifications.} 93E20, 60H10, 35K15

\addcontentsline{toc}{section}{\hspace*{1.8em}Abstract}

\section{Introduction}

It is well known that the maximum principle, namely, the necessary condition
for optimality, is an important tool in solving optimal control problems. In
1990, Peng \cite{Peng90} first obtained the global stochastic maximum
principle for the classical stochastic control systems with diffusion
coefficients including control variables. Since then, many researchers
investigate this kind of optimal control problems for various stochastic
systems (see \cite{Fuhrman-Hu06}, \cite{Fuhrman-Hu13}, \cite{Hu-Jin-Zhou12},
\cite{Peng93}, \cite{Tang98}, \cite{Tang03}, \cite{Tang-Li94} and references
therein). Recently, Hu \cite{Hu17} first introduced two adjoint equations to
obtain the global maximum principle for decoupled forward-backward stochastic
control systems and completely solved the open problem proposed by Peng in
\cite{Peng99}. Inspired by Hu \cite{Hu17}, Hu, Ji and Xue \cite{HuJiXue18}%
\ proposed a new method to obtain the first and second-order variational
equations which are essentially fully coupled forward-backward stochastic
differential equations (see \cite{HuYing-Peng95}, \cite{Ma-Yong-Protter},
\cite{Ma-WZZ}, \cite{Ma-ZZ}, \cite{Zhang17}), and derived the global maximum
principle for fully coupled forward-backward stochastic control systems.

In this paper, we study the following decoupled forward-backward stochastic
control system:%
\begin{equation}
\left\{
\begin{array}
[c]{rl}%
dX_{t}^{u}= & b(t,X_{t}^{u},u_{t})dt+\sigma(t,X_{t}^{u},u_{t})dW_{t},\\
dY_{t}^{u}= & -f(t,X_{t}^{u},Y_{t}^{u},Z_{t}^{u},u_{t})dt+\left(  Z_{t}%
^{u}\right)  ^{\intercal}dW_{t},\\
X_{0}^{u}= & x_{0},\ Y_{T}^{u}=\Phi(X_{T}^{u}),
\end{array}
\right.  \label{intro-FBSDE}%
\end{equation}
where the generator $f$ of the backward stochastic differential equation (BSDE
for short) in (\ref{intro-FBSDE}) has a quadratic growth in $z$ and the
control $u(\cdot)$ takes values in a nonempty subset of $\mathbb{R}^{k}$
($k\geq1$). Our aim is to minimize the cost functional $J(u(\cdot)):=Y_{0}%
^{u}$ and obtain the global stochastic maximum principle for
(\ref{intro-FBSDE}).

The solvability of BSDEs with quadratic generators itself is an important
research field. Kobylanski \cite{Kobylanski} first investigated the case of
bounded terminal values. After that, Briand and Hu (\cite{HuBSDEquad06},
\cite{HuBSDEquad08}) studied quadratic BSDEs with unbounded terminal values
and convex generators. Recently, numerous progress has been made (see
\cite{Bahlali}, \cite{Barr-ElKar}, \cite{Hu-Tang2018}, \cite{Tevzadze},
\cite{Xing18}) in this issue. The stochastic optimal control problems
involving BSDEs with quadratic generators have a wide range of applications in
the field of control and finance. When
\[
f(t,X_{t}^{u},Y_{t}^{u},Z_{t}^{u},u_{t})=\frac{\gamma}{2}\left\vert Z_{t}%
^{u}\right\vert ^{2}+g(t,X_{t}^{u},u_{t}),
\]
(\ref{intro-FBSDE}) becomes a risk-sensitive control problem with the
risk-sensitive parameter $\gamma$ investigated by Lim and Zhou in
\cite{Lim-Zhou} (see also \cite{Whittle90}). On the other hand,
(\ref{intro-FBSDE}) is closely related to the exponential utility maximization
problems (see \cite{Delbaen02}, \cite{El-Rouge}, \cite{Hu-Im-Mull},
\cite{Morlais}).

The backward state equation with a quadratic generator in (\ref{intro-FBSDE})
has caused great difficulty to explore the global stochastic maximum principle
for (\ref{intro-FBSDE}). Firstly, the quadratic growth in $z$ leads to that
the first and second-order variational equations are linear BSDEs with
unbounded stochastic Lipschitz coefficients involving BMO-martingales which
has been studied by Briand and Confortola \cite{Briand-SL-BSDE}. To obtain the
global stochastic maximum principle, roughly speaking, the key step is to
prove $\mathbb{E}\left[  \left(  \int_{t}^{t+\varepsilon}\left\vert \bar
{Z}_{s}\right\vert ds\right)  ^{\beta}\right]  =O\left(  \varepsilon^{\beta
}\right)  $ for $\varepsilon>0$ small enough, where $\bar{Z}(\cdot)$ is the
optimal state trajectory. In our context, we only know that $\mathbb{E}\left[
\left(  \int_{0}^{T}\left\vert \bar{Z}_{s}\right\vert ^{2}ds\right)
^{p}\right]  <\infty$ for any $p>0$. So we need $\beta\leq2$ to obtain the
above desired estimate. However, the estimate in \cite{Briand-SL-BSDE} cannot
guarantee $\beta\leq2$. To overcome this difficulty, we establish a new
estimate (see Proposition \ref{est-exp-yz-prop}) for linear BSDEs with
unbounded stochastic Lipschitz coefficients involving BMO-martingales. Based
on this estimate, we can deduce the first and second-order variational
equations of (\ref{intro-FBSDE}).

Secondly, for the multi-dimensional case, the first and second-order adjoint
equations are multi-dimensional linear BSDEs with unbounded stochastic
Lipschitz coefficients involving BMO-martingales which have been investigated
by Delbaen and Tang \cite{Delbaen-Tang2010}. In their proof, the key point to
obtain the solvability of BSDE is that the solution of the following
matrix-valued homogenous SDE:
\[
X_{t}=I_{n}+\sum_{k=1}^{d}\int_{0}^{t}A_{s}^{k}X_{s}dW_{s}^{k},\text{ }%
t\in\lbrack0,T]
\]
satisfies the reverse H\"{o}lder inequality by assuming that $A_{{}}^{k}\cdot
W^{k}\in\mathcal{\bar{H}}_{\infty}^{BMO}$ for $k=1,2,\ldots,d$, where the
space $\mathcal{\bar{H}}_{\infty}^{BMO}$ is the closure of the space
$\mathcal{H}_{\infty}$ (the set of all martingales with essentially bounded
quadratic variations) under the BMO-norm. However, due to\ Theorem 2.8 in
\cite{Delbaen-Tang2010}, the adjoint equations in our paper may not satisfy
the assumption in \cite{Delbaen-Tang2010}. Fortunately, we find both the first
and second-order adjoint equations belong to a class of specially structured
linear BSDEs (see Proposition \ref{prop-multi-linear-BSDE}) with unbounded
stochastic Lipschitz coefficients involving BMO-martingales. Then we obtain
the reverse H\"{o}lder inequality for the associated homogenous SDEs (see
Proposition \ref{prop-multi-reverseHolder}) which leads to the existence and
uniqueness of adjoint equations.

The rest of the paper is organized as follows. In section 2, preliminaries and
formulation of our problem are given. In section 3, we first establish a new
estimate for linear BSDEs with stochastic Lipschitz coefficients involving
BMO-martingales, and obtain the solvability of a class of multi-dimensional
special structured linear BSDEs with unbounded stochastic Lipschitz
coefficients involving BMO-martingales. Based on these results, the
variational equations and the adjoint equations are obtained, and a global
stochastic maximum principle is derived by spike variation method.

\section{Preliminaries and problem formulation}

Let $T\in(0,+\infty)$ be fixed and $(\Omega,\mathcal{F},\mathbb{P})$ be a complete
probability space on which a standard $d$-dimensional Brownian motion
$W=(W_{t}^{1},W_{t}^{2},\ldots,W_{t}^{d})_{t\in\lbrack0,T]}^{\intercal}$ is
defined. $\mathbb{F}:\mathbb{=}\left\{  \mathcal{F}_{t},0\leq t\leq T\right\}
$ is the $\mathbb{P}$-augmentation of the natural filtration of $W$. Denote by
$\mathbb{R}^{n}$ the $n$-dimensional real Euclidean space, $\mathbb{R}%
^{k\times n}$ the set of all $k\times n$ real matrices and $\mathbb{S}%
^{n\times n}$ the set of all $n\times n$ real symmetric matrices. The scalar
product (resp. norm) of $A=(a_{ij})$, $B=(b_{ij})\in\mathbb{R}^{k\times n}$ is
denoted by $\left\langle A,B\right\rangle =\mathrm{tr}\{AB^{\intercal}\}$
(resp. $\left\vert A\right\vert =\sqrt{\mathrm{tr}\left\{  AA^{\intercal
}\right\}  }$), where the superscript $^{\intercal}$ denotes the transpose of
vectors or matrices.

For any given $p,q\geq1$, we introduce the following spaces.

$L_{\mathcal{F}_{T}}^{p}(\Omega;\mathbb{R}^{n})$: the space of $\mathcal{F}%
_{T}$-measurable $\mathbb{R}^{n}$-valued random variables $\eta$ such that
$\mathbb{E}\left[  \left\vert \eta\right\vert ^{p}\right]  <\infty$.

$L_{\mathcal{F}_{T}}^{\infty}(\Omega;\mathbb{R}^{n})$: the space of
$\mathcal{F}_{T}$-measurable $\mathbb{R}^{n}$-valued random variables $\eta$
such that $\underset{\omega\in\Omega}{\mathrm{ess~sup}}\left\vert \eta\left(
\omega\right)  \right\vert <\infty$.

$\mathcal{M}_{\mathcal{F}}^{p,q}([0,T];\mathbb{R}^{n})$: the space of
$\mathbb{F}$-adapted processes $\varphi(\cdot)$ on $[0,T]$ such that
\[%
\begin{array}
[c]{rl}%
\left\Vert \varphi(\cdot)\right\Vert _{p,q}= & \left(  \mathbb{E}\left[
\left(  \int_{0}^{T}\left\vert \varphi_{t}\right\vert ^{p}dt\right)
^{\frac{q}{p}}\right]  \right)  ^{\frac{1}{q}}<\infty.
\end{array}
\]
In particular, we denote by $\mathcal{M}_{\mathcal{F}}^{p}([0,T];\mathbb{R}%
^{n})$ the above space when $p=q$.

$L_{\mathcal{F}}^{\infty}([0,T];\mathbb{R}^{n})$: the space of $\mathbb{F}%
$-adapted processes $\varphi(\cdot)$ on $[0,T]$ such that
\[
\left\Vert \varphi(\cdot)\right\Vert _{\infty}=\underset{(t,\omega)\in
\lbrack0,T]\times\Omega}{\mathrm{ess~sup}}\left\vert \varphi_{t}\left(
\omega\right)  \right\vert <\infty.
\]

$\mathcal{S}_{\mathcal{F}}^{p}([0,T];\mathbb{R}^{n})$: the space of continuous
processes $\varphi(\cdot)\in$ $\mathcal{M}_{\mathcal{F}}^{p}([0,T];\mathbb{R}%
^{n})$ such that $\mathbb{E}\left[  \sup\limits_{t\in\lbrack0,T]}\left\vert \varphi
_{t}\right\vert ^{p}\right]  <\infty$.

$BMO_{p}$: the space of $\mathbb{F}$-adapted and real-valued martingales $M$
such that%
\begin{equation}
\left\Vert M\right\Vert _{BMO_{p}}:=\sup_{\tau}\left\Vert \left(
\mathbb{E}\left[  \left\vert M_{T}-M_{\tau}\right\vert ^{p}\mid\mathcal{F}%
_{\tau}\right]  \right)  ^{\frac{1}{p}}\right\Vert _{\infty}<\infty,\text{
\ }p\in\lbrack1,+\infty), \label{def-BMOp}%
\end{equation}
where the supremum is taken over all stopping times $\tau\in\mathcal{[}%
0,T\mathcal{]}$.

\subsection{Some properties for BMO-martingales}

Here some properties and notations for BMO-martingales are listed. We refer
the readers to \cite{HWY}, \cite{Kazamaki} and the references therein for more details.

\begin{enumerate}
\item Let $p\in(1,+\infty)$. Then, there is a positive constant $C_{p}$ such
that, for any $\mathbb{F}$-adapted and real-valued martingales $M$, we have%
\[
\left\Vert M\right\Vert _{_{BMO_{1}}}\leq\left\Vert M\right\Vert _{BMO_{p}%
}\leq C_{p}\left\Vert M\right\Vert _{_{BMO_{1}}}.
\]
Thus, for any $p\geq1$, we write simply $BMO$ for $BMO_{p}$.

\item The energy inequality: If $M\in BMO,$ then, for any positive integer
$n$,%
\begin{equation}
\mathbb{E}[\left\langle M\right\rangle _{T}^{n}]\leq n!\left\Vert M\right\Vert
_{BMO_{2}}^{2n}. \label{energy-ineq}%
\end{equation}

\begin{remark}
Applying H\"{o}lder's inequality to (\ref{energy-ineq}), for any
$p\in(0,+\infty)$, we have%
\[
\mathbb{E}\left[  \left\langle M\right\rangle _{T}^{p}\right]  \leq\left(
\left(  \left[  p\right]  +1\right)  !\left\Vert M\right\Vert _{BMO_{2}%
}^{2\left(  \left[  p\right]  +1\right)  }\right)  ^{\frac{p}{\left[
p\right]  +1}}.
\]

\end{remark}

\item Denote by $\mathcal{E}\left(  M\right)  $ the Dol\'{e}ans-Dade exponential of
a continuous local martingale $M$, that is, $\mathcal{E}\left(  M_{t}\right)
=\exp\left\{  M_{t}-\frac{1}{2}\left\langle M\right\rangle _{t}\right\}  $ for
any $t\in\lbrack0,T].$ If $M\in BMO$, then $\mathcal{E}(M)$ is a uniformly
integrable martingale.

\item The John-Nirenberg inequality: Let $M\in BMO$. If $\delta\in\left(
0,\left\Vert M\right\Vert _{BMO_{2}}^{-2}\right)  $, then we have
\begin{equation}
\mathbb{E}\left[  \exp\left\{  \delta([\left\langle M\right\rangle
_{T}-\left\langle M\right\rangle _{\tau})\right\}  \mid\mathcal{F}_{\tau
}\right]  \leq\left(  1-\delta\left\Vert M\right\Vert _{BMO_{2}}^{2}\right)
^{-1} \label{J-N ineq}%
\end{equation}
for all stopping time $\tau\in\mathcal{[}0,T\mathcal{]}$.

\item The reverse H\"{o}lder inequality:\ Let $M\in BMO$. Denote by $p_{M}%
^{{}}$\ the positive constant linked to a BMO-martingale $M$ such that
$\Psi(p_{M}^{{}})=\left\Vert M\right\Vert _{BMO_{2}}$, where the  monotonically decreasing function%
\begin{equation}
\Psi(x)=\sqrt{1+\frac{1}{x^{2}}\ln\frac{2x-1}{2(x-1)}}-1,\text{ \ \ \ }%
x\in(1,+\infty). \label{Rp-cd-func}%
\end{equation}
We make the convention that $\Psi(+\infty):=\lim\limits_{x\rightarrow+\infty
}\Psi(x)=0$ and $p_{M}^{{}}=+\infty$ when $\left\Vert M\right\Vert _{BMO_{2}%
}=0$. One can check that $p_{M}^{{}}$ is uniquely determined. If $p\in\left(
1,p_{M}^{{}}\right)  $, then, for any stopping time $\tau\in\mathcal{[}%
0,T\mathcal{]}$,
\begin{equation}
\mathbb{E}[\left(  \mathbb{\mathcal{E}}\left(  M_{T}\right)  \right)
^{p}/\left(  \mathbb{\mathcal{E}}\left(  M_{\tau}\right)  \right)  ^{p}%
\mid\mathcal{F}_{\tau}]\leq K\left(  p,\left\Vert M\right\Vert _{BMO_{2}%
}\right)  ,\text{ \ }\mathbb{P}-a.s., \label{Rp-cd}%
\end{equation}
where
\begin{equation}%
\begin{array}
[c]{rl}
& K\left(  p,\left\Vert M\right\Vert _{BMO_{2}}\right)  \\
= & 2\left(  1-\dfrac{2p-2}{2p-1}\exp\left\{  p^{2}\left[  \left\Vert
M\right\Vert _{BMO_{2}}^{2}+2\left\Vert M\right\Vert _{BMO_{2}}\right]
\right\}  \right)  ^{-1}.
\end{array}
\label{Rp-const}%
\end{equation}

\item $p_{M}^{\ast}$: the conjugate exponent of $p_{M}^{{}}$, that is,
$\left(  p_{M}^{{}}\right)  ^{-1}+\left(  p_{M}^{\ast}\right)  ^{-1}=1$.

$H\cdot W$: $H$ is an $\mathbb{F}$-adapted process and $H\cdot W$ is the
stochastic integral of $H$ with respect to $W$. If $H\cdot W\in BMO$, then we
write simply $p_{H}^{{}}$ for $p_{H\cdot W}^{{}}$ and $p_{H}^{\ast}$ for
$p_{H\cdot W}^{\ast}$ without ambiguity.
\end{enumerate}

\subsection{Problem formulation}

Consider the following forward-backward stochastic control system:
\begin{equation}
\left\{
\begin{array}
[c]{rl}%
dX_{t}^{u}= & b(t,X_{t}^{u},u_{t})dt+\sigma(t,X_{t}^{u},u_{t})dW_{t},\\
dY_{t}^{u}= & -f(t,X_{t}^{u},Y_{t}^{u},Z_{t}^{u},u_{t})dt+\left(  Z_{t}%
^{u}\right)  ^{\intercal}dW_{t},\\
X_{0}^{u}= & x_{0},\ Y_{T}^{u}=\Phi(X_{T}^{u}),
\end{array}
\right.  \label{state-eq}%
\end{equation}
where $b:[0,T]\times\mathbb{R}^{n}\times U\longmapsto\mathbb{R}^{n}$,
$\sigma:[0,T]\times\mathbb{R}^{n}\times U\longmapsto\mathbb{R}^{n\times d}$,
$f:[0,T]\times\mathbb{R}^{n}\times\mathbb{R}\times\mathbb{R}^{d}\times
U\longmapsto\mathbb{R}$, $\Phi:\mathbb{R}^{n}\longmapsto\mathbb{R}$, $x_{0} \in \mathbb{R}^{n}$ and
the control domain $U$ is a nonempty subset of $\mathbb{R}^{k}$. Set
\[%
\begin{array}
[c]{l}%
b(\cdot)=\left(  b^{1}(\cdot),b^{2}(\cdot),\ldots,b_{{}}^{n}(\cdot)\right)
^{\intercal}\in\mathbb{R}^{n},\\
\sigma(\cdot)=\left(  \sigma^{1}(\cdot),\sigma^{2}(\cdot),\ldots,\sigma
^{d}(\cdot)\right)  \in\mathbb{R}^{n\times d},\\
\sigma^{i}(\cdot)=\left(  \sigma^{1i}(\cdot),\sigma^{2i}(\cdot),\ldots
,\sigma_{{}}^{ni}(\cdot)\right)  ^{\intercal}\in\mathbb{R}^{n}\text{ for
}i=1,2,\ldots,d.
\end{array}
\]
An admissible control $u(\cdot)$ is an $\mathbb{F}$-adapted process with
values in $U$ such that%
\begin{equation}
\mathbb{E}\left[  \int_{0}^{T}\left\vert u_{t}\right\vert ^{p}dt\right]
<\infty\text{ for any }p>0\text{.} \label{control-integrable}%
\end{equation}
Denote by $\mathcal{U}[0,T]$ the admissible control set. We impose the
following assumptions on the coefficients of the control system
(\ref{state-eq}).

\begin{assumption}
\label{assum-2} (i) $\Phi$, $\Phi_{x}$, $\Phi_{xx}$ are continuous and bounded.

(ii) For $\psi=b$ and $\sigma$, $\psi$, $\psi_{x}$ and $\psi_{xx}$ are
continuous in $(x,u)$; $\psi_{x}$ and $\psi_{xx}$ are bounded; there exists a
constant $L_{1}>0$ such that%
\[
\left\vert b(t,0,u)\right\vert +\left\vert \sigma(t,0,u)\right\vert \leq
L_{1}(1+|u|).
\]

(iii) $f$, $f_{x}$, $f_{y}$, $f_{z}$, $f_{xx}$, $f_{xy}$, $f_{yy}$ , $f_{xz}$,
$f_{yz}$, $f_{zz}$ are continuous in $(x,y,z,u)$; $f_{x}$, $f_{y}$, $f_{xx}$,
$f_{xy}$, $f_{yy}$ , $f_{xz}$, $f_{yz}$ and $f_{zz}$ are bounded; there exist
positive constants $\alpha$, $\gamma$,$\ L_{2}$ and $L_{3}$ such that%
\[%
\begin{array}
[c]{l}%
\left\vert f(t,x,0,0,u)\right\vert \leq\alpha,\\
\left\vert f(t,x,y,z,u_{1})-f(t,x,y,z,u_{2})\right\vert \leq L_{2}%
(1+\left\vert y\right\vert +\left\vert z\right\vert ),\\
\left\vert f_{z}(t,x,y,z,u)\right\vert \leq L_{3}+\gamma\left\vert
z\right\vert .
\end{array}
\]

\end{assumption}

From the existence result (Proposition 3) in \cite{HuBSDEquad08} and the
uniqueness result (Lemma 2.1) in \cite{Hu-Tang2018}, we have the following result:

\begin{theorem}
\label{state-eq-exist-th} Let Assumption \ref{assum-2} hold. Then, for any
$u(\cdot)\in\mathcal{U}[0,T]$ and $p>1$, the state equation (\ref{state-eq})
admits a unique solution $(X_{{}}^{u}\left(  \cdot\right)  ,Y_{{}}^{u}\left(
\cdot\right)  ,Z_{{}}^{u}\left(  \cdot\right)  )\in\mathcal{S}_{\mathcal{F}%
}^{p}([0,T];\mathbb{R}^{n})\times L_{\mathcal{F}}^{\infty}([0,T];\mathbb{R}%
)\times$ $\mathcal{M}_{\mathcal{F}}^{2,p}([0,T];\mathbb{R}^{d})$ such that
$Z_{{}}^{u}\cdot W\in BMO$. Furthermore, we have the following estimate:%
\begin{equation}
\left\{
\begin{array}
[c]{l}%
\begin{array}
[c]{rl}%
\left\Vert X_{{}}^{u}\right\Vert _{\mathcal{S}^{p}}^{p}\leq & C_{1}\left\{
\left\vert x_{0}\right\vert ^{p}+\mathbb{E}\left[  \left(  \int_{0}%
^{T}\left\vert b(s,0,u_{s})\right\vert ds\right)  ^{p}\right.  \right.  \\
& \ \ \ \ \ \ \ \ +\left.  \left.  \left(  \int_{0}^{T}\left\vert
\sigma(s,0,u_{s})\right\vert ^{2}ds\right)  ^{\frac{p}{2}}\right]  \right\}  ,
\end{array}
\\
\left\Vert Y_{{}}^{u}\right\Vert _{\infty}+\left\Vert Z_{{}}^{u}\cdot
W\right\Vert _{BMO_{2}}\leq C_{2},
\end{array}
\right.  \label{est-quad-qfbsde}%
\end{equation}
where $C_{1}$ depends on $p$, $T$, $\left\Vert b_{x}\right\Vert _{\infty}$ and
$\left\Vert \sigma_{x}\right\Vert _{\infty}$, and $C_{2}$ depends on $\alpha,$
$\gamma,$ $T,$ $L_{3},$ $\left\Vert \Phi\right\Vert _{\infty}$ and $\left\Vert
f_{y}\right\Vert _{\infty}$.
\end{theorem}

Our optimal control problem is to minimize the cost functional $J(u(\cdot
)):=Y_{0}^{u}$ over $\mathcal{U}[0,T]$:%
\begin{equation}
\underset{u(\cdot)\in\mathcal{U}[0,T]}{\inf}J(u(\cdot)). \label{obje-eq}%
\end{equation}

\section{Stochastic maximum principle}

We derive the maximum principle for the optimization problem (\ref{obje-eq})
in this section. For the simplicity of presentation, the constant $C$ will change
from line to line in our proof.

\subsection{Linear BSDEs with unbounded coefficients}

We first consider the following one-dimensional linear BSDE with unbounded
coefficients:%
\begin{equation}
Y_{t}=\xi+\int_{t}^{T}\left(  \lambda_{s}Y_{s}+\mu_{s}^{\intercal}%
Z_{s}+\varphi_{s}\right)  ds-\int_{t}^{T}Z_{s}^{\intercal}dW_{s},
\label{linear-bsde}%
\end{equation}
where the Brownian motion is $d$-dimensional.

The following result is a direct application of Corollary 9 and Theorem 10
in \cite{Briand-SL-BSDE} to the BSDE (\ref{linear-bsde}).

\begin{proposition}
\label{est-linear-prop} Suppose that $\left(  \xi,\varphi\right)  \in
\bigcap_{\beta>1}\left(  L_{\mathcal{F}_{T}}^{\beta}(\Omega;\mathbb{R}%
)\times\mathcal{M}_{\mathcal{F}}^{1,\beta}([0,T];\mathbb{R})\right)  $,
$\left\Vert \lambda\right\Vert _{\infty}<\infty$ and $\mu\cdot W\in BMO$.
Then, there exists a unique solution $(Y,Z)\in\mathcal{S}_{\mathcal{F}}%
^{\beta}([0,T];\mathbb{R})\times\mathcal{M}_{\mathcal{F}}^{2,\beta
}([0,T];\mathbb{R}^{d})$ to the BSDE (\ref{linear-bsde}) for all $\beta>1$.
Moreover, for any $\beta_{0}>p_{\mu}^{\ast}$ and $\beta\in\left(  p_{\mu
}^{\ast},\beta_{0}\right)  $, we have
\begin{equation}%
\begin{array}
[c]{rl}%
\mathbb{E}\left[  \sup\limits_{t\in\lbrack0,T]}\left\vert Y_{t}\right\vert
^{\beta}+\left(  \int_{0}^{T}\left\vert Z_{t}\right\vert ^{2}dt\right)
^{\frac{\beta}{2}}\right]  \leq & C\left(  \mathbb{E}\left[  \left\vert
\xi\right\vert ^{\beta_{0}}+\left(  \int_{0}^{T}\left\vert \varphi
_{t}\right\vert dt\right)  ^{\beta_{0}}\right]  \right)  ^{\frac{\beta}%
{\beta_{0}}},
\end{array}
\label{est-linear-data}%
\end{equation}
where $C\ $is a constant depending on\ $\beta,$ $\beta_{0}$, $T$, $\left\Vert \lambda\right\Vert _{\infty}$ and $\left\Vert
\mu\cdot W\right\Vert _{BMO_{2}}$, and increasing with respect to $\left\Vert
\mu\cdot W\right\Vert _{BMO_{2}}$.
\end{proposition}

As it is emphasized in the introduction, we need $\beta_{0}<2$ in the estimate
of the first and second-order variational equations. But $p_{\mu}^{\ast}\geq2$
implies $\beta_{0}>2$ in the above proposition. Hence, we give a new estimate
for the linear BSDE (\ref{linear-bsde}) in the following proposition to guarantee
$\beta_{0}<2$.

\begin{proposition}
\label{est-exp-yz-prop} Let the assumptions in Proposition
\ref{est-linear-prop} hold, and $(Y,Z)\in\mathcal{S}_{\mathcal{F}}^{\beta
}([0,T];\mathbb{R})\times\mathcal{M}_{\mathcal{F}}^{2,\beta}([0,T];\mathbb{R}%
^{d})$ be the unique solution to the linear BSDE (\ref{linear-bsde}) for all
$\beta>1$. Then, for any $\beta\in\left(  1\vee2p_{\mu}^{-1},2\right)  $ and
$\beta_{0}\in\left(  \beta,2\right)  $, we have%
\begin{equation}
\mathbb{E}\left[  \sup\limits_{t\in\lbrack0,T]}\Gamma_{t}\left\vert
Y_{t}\right\vert ^{\beta}+\left(  \int_{0}^{T}\left(  \Gamma_{t}\right)
^{\frac{2}{\beta}}\left\vert Z_{t}\right\vert ^{2}dt\right)  ^{\frac{\beta}%
{2}}\right]  \leq C\left(  \mathbb{E}\left[  \left(  \Gamma_{t}\right)
^{\frac{\beta_{0}}{\beta}}\left\vert \xi\right\vert ^{\beta_{0}}+\left(
\int_{0}^{T}\Gamma_{t}^{\frac{1}{\beta}}\left\vert \varphi_{t}\right\vert
dt\right)  ^{\beta_{0}}\right]  \right)  ^{\frac{\beta}{\beta_{0}}%
},\label{est-exp-yz}%
\end{equation}
where $\Gamma_{t}=\mathcal{E}\left(  \int_{0}^{t}\mu_{s}^{\intercal}%
dW_{s}\right)  $ for $t\in\lbrack0,T],$ and the constant $C$ depends on
$\beta$, $\beta_{0}$, $T$, $\left\Vert \lambda\right\Vert _{\mathcal{\infty}}$
and $\left\Vert \mu\cdot W\right\Vert _{BMO_{2}}$.
\end{proposition}

\begin{proof}
Set $\tilde{\Gamma}_{t}=\exp\left\{  \int_{0}^{t}\lambda_{s}ds\right\}
\Gamma_{t}$ for $t\in\lbrack0,T]$ and note that $\tilde{\Gamma
}_{t}\leq e^{\left\Vert \lambda\right\Vert _{\mathcal{\infty}}T}\Gamma_{t}$.
Applying It\^{o}'s formula to $\tilde{\Gamma}Y$, we get%
\begin{equation}
\tilde{\Gamma}_{t}Y_{t}=\tilde{\Gamma}_{T}\xi+\int_{t}^{T}\tilde{\Gamma}%
_{s}\varphi_{s}ds-\int_{t}^{T}\tilde{\Gamma}_{s}(Y_{s}\mu_{s}^{\intercal
}+Z_{s}^{\intercal})dW_{s}.\label{gamma-y}%
\end{equation}
By reverse H\"{o}lder's inequality, we have $\mathbb{E}\left[
\left(  \int_{0}^{T}\left\vert \tilde{\Gamma}_{t}(Y_{t}\mu_{t}+Z_{t}%
)\right\vert ^{2}dt\right)  ^{\frac{1}{2}}\right]  <\infty$. Then, by taking
conditional expectation on both sides of (\ref{gamma-y}), we obtain%
\[%
\begin{array}
[c]{rl}%
Y_{t}= & \mathbb{E}\left[  \left(  \frac{\tilde{\Gamma}_{T}}{\tilde{\Gamma
}_{t}}\right)  \xi+\int_{t}^{T}\left(  \frac{\tilde{\Gamma}_{s}}{\tilde
{\Gamma}_{t}}\right)  \varphi_{s}ds\mid\mathcal{F}_{t}\right]  .
\end{array}
\]
For the simplicity of writing, we denote $\mathbb{E}\left[  \cdot
\mid\mathcal{F}_{t}\right]  $ by $\mathbb{E}_{t}\left[  \cdot\right]  $, and
set $\left(  \Gamma_{s}^{t}\right)  _{s\in\lbrack t,T]}=\left(  \frac
{\Gamma_{s}}{\Gamma_{t}}\right)  _{s\in\lbrack t,T]}$ for any fixed
$t\in\lbrack0,T]$. For any $\beta\in\left(  1\vee2p_{\mu}^{-1},2\right)  $ and
$\beta^{\ast}=\beta(\beta-1)^{-1}$, by H\"{o}lder's inequality,%
\[%
\begin{array}
[c]{cl}%
\left\vert Y_{t}\right\vert  & \leq e^{\left\Vert \lambda\right\Vert
_{\mathcal{\infty}}T}\mathbb{E}_{t}\left[  \Gamma_{T}^{t}\left\vert
\xi\right\vert +\int_{t}^{T}\Gamma_{s}^{t}\left\vert \varphi_{s}\right\vert
ds\right]  \\
& \leq e^{\left\Vert \lambda\right\Vert _{\mathcal{\infty}}T}\mathbb{E}%
_{t}\left[  \left(  \Gamma_{T}^{t}\right)  ^{\frac{1}{\beta^{\ast}}+\frac
{1}{\beta}}\left\vert \xi\right\vert +\left(  \sup\limits_{s\in\lbrack
t,T]}\Gamma_{s}^{t}\right)  ^{\frac{1}{\beta^{\ast}}}\int_{t}^{T}\left(
\Gamma_{s}^{t}\right)  ^{\frac{1}{\beta}}\left\vert \varphi_{s}\right\vert
ds\right]  \\
& \leq e^{\left\Vert \lambda\right\Vert _{\mathcal{\infty}}T}\left(
\mathbb{E}_{t}\left[  \sup\limits_{s\in\lbrack t,T]}\Gamma_{s}^{t}\right]
\right)  ^{\frac{1}{\beta^{\ast}}}\left\{  \left(  \mathbb{E}_{t}\left[
\Gamma_{T}^{t}\left\vert \xi\right\vert ^{\beta}\right]  \right)  ^{\frac
{1}{\beta}}\right.  \\
& \ \ +\left.  \left(  \mathbb{E}_{t}\left[  \left(  \int_{t}^{T}\left(
\Gamma_{s}^{t}\right)  ^{\frac{1}{\beta}}\left\vert \varphi_{s}\right\vert
ds\right)  ^{\beta}\right]  \right)  ^{\frac{1}{\beta}}\right\}  .
\end{array}
\]
Thanks to Doob's inequality and reverse H\"{o}lder's inequality, we have%
\[%
\begin{array}
[c]{rcl}%
\mathbb{E}_{t}\left[  \sup\limits_{s\in\lbrack t,T]}\Gamma_{s}^{t}\right]
\leq & \left(  \mathbb{E}_{t}\left[  \sup\limits_{s\in\lbrack t,T]}\left(
\Gamma_{s}^{t}\right)  ^{\frac{1+p_{\mu}}{2}}\right]  \right)  ^{\frac
{2}{1+p_{\mu}}}\leq & C\left(  \mathbb{E}_{t}\left[  \left(  \Gamma_{T}%
^{t}\right)  ^{\frac{1+p_{\mu}}{2}}\right]  \right)  ^{\frac{2}{1+p_{\mu}}%
}\leq C,
\end{array}
\]
where $C$ depends on $\left\Vert \mu\cdot W\right\Vert _{BMO_{2}}$. Thus, we
obtain%
\[%
\begin{array}
[c]{cc}%
\left\vert Y_{t}\right\vert \leq & C\left\{  \left(  \mathbb{E}_{t}\left[
\Gamma_{T}^{t}\left\vert \xi\right\vert ^{\beta}\right]  \right)  ^{\frac
{1}{\beta}}+\left(  \mathbb{E}_{t}\left[  \left(  \int_{t}^{T}\left(
\Gamma_{s}^{t}\right)  ^{\frac{1}{\beta}}\left\vert \varphi_{s}\right\vert
ds\right)  ^{\beta}\right]  \right)  ^{\frac{1}{\beta}}\right\}  ,
\end{array}
\]
which implies that%
\[%
\begin{array}
[c]{rl}%
\sup\limits_{t\in\lbrack0,T]}\Gamma_{t}\left\vert Y_{t}\right\vert ^{\beta
}\leq & C\sup\limits_{t\in\lbrack0,T]}\mathbb{E}_{t}\left[  \Gamma
_{T}\left\vert \xi\right\vert ^{\beta}+\left(  \int_{0}^{T}\Gamma_{s}%
^{\frac{1}{\beta}}\left\vert \varphi_{s}\right\vert ds\right)  ^{\beta
}\right]  ,
\end{array}
\]
where $C$ depends on $\beta$, $T$, $\left\Vert \lambda\right\Vert _{\mathcal{\infty}}$
and $\left\Vert \mu\cdot W\right\Vert _{BMO_{2}}$.

Set $M_{t}=\mathbb{E}_{t}\left[  \Gamma_{T}\left\vert \xi\right\vert ^{\beta
}+\left(  \int_{0}^{T}\left(  \Gamma_{s}\right)  ^{\frac{1}{\beta}}\left\vert
\varphi_{s}\right\vert ds\right)  ^{\beta}\right]  $
for $t\in\lbrack0,T]$. For any $\beta_{0}\in\left(  \beta,2\right)  $, by H\"{o}lder's inequality,
Doob's inequality and reverse H\"{o}lder's inequality, we have
\[%
\begin{array}
[c]{rl}%
\mathbb{E}\left[  \left\vert M_{T}\right\vert ^{\frac{\beta_{0}}{\beta}%
}\right]  \leq & C\mathbb{E}\left[  \sup\limits_{s\in\lbrack0,T]}\left(
\Gamma_{s}\right)  ^{\frac{\beta_{0}}{\beta}}\left(  \left\vert \xi\right\vert
^{\beta_{0}}+\left(  \int_{0}^{T}\left\vert \varphi_{s}\right\vert ds\right)
^{\beta_{0}}\right)  \right]  \\
\leq & C\left(  \mathbb{E}\left[  \sup\limits_{s\in\lbrack0,T]}\left(
\Gamma_{s}\right)  ^{\frac{2}{\beta}}\right]  \right)  ^{\frac{\beta_{0}}{2}%
}\left(  \mathbb{E}\left[  \left\vert \xi\right\vert ^{\frac{2\beta_{0}%
}{2-\beta_{0}}}+\left(  \int_{0}^{T}\left\vert \varphi_{s}\right\vert
ds\right)  ^{\frac{2\beta_{0}}{2-\beta_{0}}}\right]  \right)  ^{\frac
{2-\beta_{0}}{2}}\\
\leq & C\left(  \mathbb{E}\left[  \Gamma_{T}^{\frac{2}{\beta}}\right]
\right)  ^{\frac{\beta_{0}}{2}}\left(  \mathbb{E}\left[  \left\vert
\xi\right\vert ^{\frac{2\beta_{0}}{2-\beta_{0}}}+\left(  \int_{0}%
^{T}\left\vert \varphi_{s}\right\vert ds\right)  ^{\frac{2\beta_{0}}%
{2-\beta_{0}}}\right]  \right)  ^{\frac{2-\beta_{0}}{2}}\\
< & \infty.
\end{array}
\]
Thus, following from Doob's inequality, we
have%
\begin{equation}
\mathbb{E}\left[  \left(  \sup\limits_{t\in\lbrack0,T]}\Gamma_{t}\left\vert
Y_{t}\right\vert ^{\beta}\right)  ^{\frac{\beta_{0}}{\beta}}\right]  \leq
C\mathbb{E}\left[  \left(  \sup\limits_{t\in\lbrack0,T]}\left\vert
M_{t}\right\vert \right)  ^{\frac{\beta_{0}}{\beta}}\right]  \leq
C\mathbb{E}\left[  \left\vert M_{T}\right\vert ^{\frac{\beta_{0}}{\beta}%
}\right]< \infty  ,\label{est-exp-y}%
\end{equation}
where $C$ depends on $\beta$, $\beta_{0}$, $T$, $\left\Vert \lambda\right\Vert _{\mathcal{\infty}}$
and $\left\Vert \mu\cdot W\right\Vert _{BMO_{2}}$.

Now we estimate $\mathbb{E}\left[  \left(  \int_{0}^{T}\Gamma_{t}^{\frac{2}{\beta}}\left\vert
Z_{t}\right\vert ^{2}dt\right)  ^{\frac{\beta}{2}}\right]  $. Applying It\^{o}'s formula to $\tilde{\Gamma
}^{\frac{2}{\beta}}\left\vert Y\right\vert ^{2}$, we get%
\begin{equation}%
\begin{array}
[c]{l}%
\left\vert Y_{0}\right\vert ^{2}+\int_{0}^{T}\tilde{\Gamma}_{t}^{\frac
{2}{\beta}}\left\vert Z_{t}\right\vert ^{2}dt+\frac{2}{\beta}(\frac{2}{\beta
}-1)\int_{0}^{T}\tilde{\Gamma}_{t}^{\frac{2}{\beta}}\left\vert \mu
_{t}\right\vert ^{2}\left\vert Y_{t}\right\vert ^{2}dt\\
=\tilde{\Gamma}_{T}^{\frac{2}{\beta}}\left\vert \xi\right\vert ^{2}%
+(2-\frac{2}{\beta})\int_{0}^{T}\tilde{\Gamma}_{t}^{\frac{2}{\beta}}%
\lambda_{t}\left\vert Y_{t}\right\vert ^{2}dt+(2-\frac{4}{\beta})\int_{0}%
^{T}\tilde{\Gamma}_{t}^{\frac{2}{\beta}}Y_{t}\mu_{t}^{\intercal}Z_{t}dt\\
\ \ +2\int_{0}^{T}\tilde{\Gamma}_{t}^{\frac{2}{\beta}}Y_{t}\varphi_{t}%
dt-\frac{2}{\beta}\int_{0}^{T}\tilde{\Gamma}_{t}^{\frac{2}{\beta}}\left\vert
Y_{t}\right\vert ^{2}\mu_{t}^{\intercal}dW_{t}-2\int_{0}^{T}\tilde{\Gamma}%
_{t}^{\frac{2}{\beta}}Y_{t}Z_{t}^{\intercal}dW_{t}.
\end{array}
\label{Ito-formula-gammaZ}%
\end{equation}
Recall that $\beta\in\left(  1\vee2p_{\mu}^{-1},2\right)  $. From
(\ref{Ito-formula-gammaZ}), by using the inequality
\begin{equation}
\left\vert a+b\right\vert ^{\frac{\beta}{2}}\leq\left\vert a\right\vert
^{\frac{\beta}{2}}+\left\vert b\right\vert ^{\frac{\beta}{2}},\quad
a,b\in\mathbb{R},\label{ess-ineq}%
\end{equation}
and the B-D-G inequality, we get
\begin{equation}%
\begin{array}
[c]{rl}
& \mathbb{E}\left[  \left(  \int_{0}^{T}\tilde{\Gamma}_{t}^{\frac{2}{\beta}%
}\left\vert Z_{t}\right\vert ^{2}dt\right)  ^{\frac{\beta}{2}}\right]  \\
\leq & \mathbb{E}\left[  \sup\limits_{t\in\lbrack0,T]}\tilde{\Gamma}%
_{T}\left\vert Y_{t}\right\vert ^{\beta}\right]  +(2-\frac{2}{\beta}%
)^{\frac{\beta}{2}}\mathbb{E}\left[  \left(  \int_{0}^{T}\tilde{\Gamma}%
_{t}^{\frac{2}{\beta}}\lambda_{t}\left\vert Y_{t}\right\vert ^{2}dt\right)
^{\frac{\beta}{2}}\right]  \\
& +(\frac{4}{\beta}-2)^{\frac{\beta}{2}}\mathbb{E}\left[  \left(  \int_{0}%
^{T}\tilde{\Gamma}_{t}^{\frac{2}{\beta}}\left\vert Y_{t}\right\vert \left\vert
\mu_{t}\right\vert \left\vert Z_{t}\right\vert dt\right)  ^{\frac{\beta}{2}%
}\right]  +2^{\frac{\beta}{2}}\mathbb{E}\left[  \left(  \int_{0}^{T}%
\tilde{\Gamma}_{t}^{\frac{2}{\beta}}\left\vert Y_{t}\right\vert \left\vert
\varphi_{t}\right\vert dt\right)  ^{\frac{\beta}{2}}\right]  \\
& +3\left(  \frac{2}{\beta}\right)  ^{\frac{\beta}{2}}\mathbb{E}\left[
\left(  \int_{0}^{T}\tilde{\Gamma}_{t}^{\frac{4}{\beta}}\left\vert
Y_{t}\right\vert ^{4}\left\vert \mu_{t}\right\vert ^{2}dt\right)
^{\frac{\beta}{4}}\right]  +3\cdot2^{\frac{\beta}{2}}\mathbb{E}\left[  \left(
\int_{0}^{T}\tilde{\Gamma}_{t}^{\frac{4}{\beta}}\left\vert Y_{t}\right\vert
^{2}\left\vert Z_{t}\right\vert ^{2}dt\right)  ^{\frac{\beta}{4}}\right]  .
\end{array}
\label{Expect-Gamma-half-beta}%
\end{equation}
Since
\[
\int_{0}^{T}\tilde{\Gamma}_{t}^{\frac{2}{\beta}}\left\vert Y_{t}\right\vert
\left\vert \mu_{t}\right\vert \left\vert Z_{t}\right\vert dt\leq(4\delta
)^{-1}\int_{0}^{T}\tilde{\Gamma}_{t}^{\frac{2}{\beta}}\left\vert
Y_{t}\right\vert ^{2}\left\vert \mu_{t}\right\vert ^{2}dt+\delta\int_{0}%
^{T}\tilde{\Gamma}_{t}^{\frac{2}{\beta}}\left\vert Z_{t}\right\vert ^{2}dt,
\]%
\[
\left(  \int_{0}^{T}\tilde{\Gamma}_{t}^{\frac{4}{\beta}}\left\vert
Y_{t}\right\vert ^{2}\left\vert Z_{t}\right\vert ^{2}dt\right)  ^{\frac{\beta
}{4}}\leq(4\delta)^{-1}\sup\limits_{t\in\lbrack0,T]}\tilde{\Gamma}%
_{t}\left\vert Y_{t}\right\vert ^{\beta}+\delta\left(  \int_{0}^{T}%
\tilde{\Gamma}_{t}^{\frac{2}{\beta}}\left\vert Z_{t}\right\vert ^{2}dt\right)
^{\frac{\beta}{2}}%
\]
with $\delta=1\wedge\left\{  2\left[  (\frac{4}{\beta}-2)^{\frac{\beta}{2}%
}+3\cdot2^{\frac{\beta}{2}}\right]  \right\}  ^{-\frac{2}{\beta}}$, and
$e^{-\left\Vert \lambda\right\Vert _{\infty}T}\Gamma_{t}\leq\tilde{\Gamma}%
_{t}\leq e^{\left\Vert \lambda\right\Vert _{\infty}T}\Gamma_{t}$, we have
\[%
\begin{array}
[c]{rl}
& \mathbb{E}\left[  \left(  \int_{0}^{T}\Gamma_{t}^{\frac{2}{\beta}}\left\vert
Z_{t}\right\vert ^{2}dt\right)  ^{\frac{\beta}{2}}\right]  \\
\leq & C\mathbb{E}\left[  \sup\limits_{t\in\lbrack0,T]}\Gamma_{t}\left\vert
Y_{t}\right\vert ^{\beta}+\left(  \int_{0}^{T}\left(  \Gamma_{t}\right)
^{\frac{1}{\beta}}\left\vert \varphi_{t}\right\vert dt\right)  ^{\beta
}+\left(  \sup\limits_{t\in\lbrack0,T]}\Gamma_{t}\left\vert Y_{t}\right\vert
^{\beta}\right)  \left(  \int_{0}^{T}\left\vert \mu_{t}\right\vert
^{2}dt\right)  ^{\frac{\beta}{2}}\right]  ,
\end{array}
\]
where $C$ depends on $\beta,T$ and $\left\Vert \lambda\right\Vert
_{\mathcal{\infty}}$. Following from H\"{o}lder's inequality and the energy
inequality for $\mu\cdot W$, we further get%
\begin{equation}
\mathbb{E}\left[  \left(  \int_{0}^{T}\Gamma_{t}^{\frac{2}{\beta}}\left\vert
Z_{t}\right\vert ^{2}dt\right)  ^{\frac{\beta}{2}}\right]  \leq C\left(
\mathbb{E}\left[  \left(  \sup\limits_{t\in\lbrack0,T]}\Gamma_{t}\left\vert
Y_{t}\right\vert ^{\beta}\right)  ^{\frac{\beta_{0}}{\beta}}+\left(  \int%
_{0}^{T}\left(  \Gamma_{t}\right)  ^{\frac{1}{\beta}}\left\vert \varphi
_{t}\right\vert dt\right)  ^{\beta_{0}}\right]  \right)  ^{\frac{\beta}%
{\beta_{0}}},\label{est-exp-z}%
\end{equation}
where $C$ depends on $\beta,$ $\beta_{0},$ $T$, $\left\Vert \lambda\right\Vert
_{\mathcal{\infty}}$ and $\left\Vert \mu\cdot W\right\Vert _{BMO_{2}}$.
Combining (\ref{est-exp-z}) with (\ref{est-exp-y}), we obtain
(\ref{est-exp-yz}) by H\"{o}lder's inequality.
\end{proof}

\bigskip

Now we consider the following $n$-dimensional linear BSDE with unbounded
coefficients:%
\begin{equation}
Y_{t}=\xi+\int_{t}^{T}\left\{  A_{s}^{\intercal}Y_{s}+\sum\limits_{i=1}%
^{d}\left(  \beta_{s}^{i}I_{n}+C_{s}^{i}\right)  ^{\intercal}Z_{s}^{i}%
+f_{s}\right\}  ds-\sum\limits_{i=1}^{d}\int_{t}^{T}Z_{s}^{i}dW_{s}^{i},\text{
}t\in\lbrack0,T], \label{multi-linear-BSDE}%
\end{equation}
where $A,$ $C^{i}:[0,T]\times\Omega\longmapsto\mathbb{R}_{{}}^{n\times n}$ and
$\beta^{i}:[0,T]\times\Omega\longmapsto\mathbb{R}$ are $\mathbb{F}$-adapted
processes for $i=1,2,\ldots,d$, and $I_{n}$ is the $n\times n$ identity
matrix. In order to obtain the solvability of the above linear BSDE, we need
to study the following matrix-valued linear SDE:%
\begin{equation}
X_{t}=I_{n}+\int_{0}^{t}A_{s}X_{s}ds+\sum_{i=1}^{d}\int_{0}^{t}\left(
\beta_{s}^{i}I_{n}+C_{s}^{i}\right)  X_{s}dW_{s}^{i},\text{ \ }t\in
\lbrack0,T]. \label{multi-linear-SDE}%
\end{equation}

\begin{proposition}
\label{prop-multi-reverseHolder} Assume that
\[
\left\Vert \left(  \left\vert A\right\vert +\sum_{i=1}^{d}\left\vert \beta
^{i}\right\vert \right)  \cdot W^{1}\right\Vert _{BMO_{2}}+\sum_{i=1}%
^{d}\left\Vert C_{{}}^{i}\right\Vert _{\infty}<\infty.
\]
Denote by $\bar{p}$ the constant such that
\[
\Psi(\bar{p})=\left(  2\sum\limits_{i=1}^{d}\left(  n\left\Vert \beta^{i}\cdot
W^{i}\right\Vert _{BMO_{2}}^{2}+\left\Vert C_{{}}^{i}\right\Vert _{\infty}%
^{2}T\right)  \right)  ^{\frac{1}{2}},
\]
where $\Psi(\cdot)$ is defined in (\ref{Rp-cd-func}). Then, there exists a
unique strong solution $X$ to (\ref{multi-linear-SDE}). Moreover, $X$
satisfies the reverse H\"{o}lder inequality for $p\in\left(  1,\bar{p}\right)
$, that is,%
\[
\mathbb{E}\left[  \sup\limits_{s\in\lbrack t,T]}\left\vert X_{s}X_{t}%
^{-1}\right\vert ^{p}\mid\mathcal{F}_{t}\right]  \leq C\text{ for any }%
t\in\lbrack0,T],
\]
where $C$ is a constant depending on $p,$ $T,$ $\left\Vert \left(  \left\vert
A\right\vert +\sum_{i=1}^{d}\left\vert \beta^{i}\right\vert \right)  \cdot
W^{1}\right\Vert _{BMO_{2}}$ and $\sum_{i=1}^{d}\left\Vert C_{{}}%
^{i}\right\Vert _{\infty}$, and $X^{-1}$ is the inverse of $X$.
\end{proposition}

\begin{proof}
Since the coefficients satisfy the conditions in the basic
theorem in \cite{Gal'Chuk1978} on pp. 756-757 to the underlying
semi-martingale $((%
\overbrace{1,1,\cdot\cdot\cdot,1}^{d}%
)_{{}}^{\intercal}t+W_{t})_{t\in\lbrack0,T]}^{{}}$ (see also Lemma 7.1 in
\cite{Tang03}), (\ref{multi-linear-SDE}) has a unique strong
solution. For any $t\in\lbrack0,T]$, set $D_{t}^{i}=\beta_{t}^{i}I_{n}%
+C_{t}^{i}$ for $i=1,2,\ldots,d$. Consider the following matrix-valued SDE:%
\begin{equation}
\Lambda_{t}=I_{n}+%
{\displaystyle\int_{0}^{t}}
\Lambda_{s}\left[  -A_{s}+\sum_{i=1}^{d}\left(  D_{s}^{i}\right)  ^{2}\right]
ds-\sum_{i=1}^{d}%
{\displaystyle\int_{0}^{t}}
\Lambda_{s}D_{s}^{i}dW_{s}^{i},\text{ \ }t\in\lbrack
0,T].\label{multi-linear-SDE-inverse}%
\end{equation}
Similar to (\ref{multi-linear-SDE}), the above
SDE has a unique strong solution $\Lambda$. By applying It\^{o}'s formula to
$\Lambda_{t}X_{t}$, we obtain $\Lambda_{t}X_{t}=I_{n}$ which
implies $X_{t}^{-1}=\Lambda_{t}$. For each fixed $t\in\lbrack0,T]$, set
$X_{s}^{t}=X_{s}X_{t}^{-1}$ for $s\in\lbrack t,T]$, then
$X^{t}$ satisfies the following SDE:%
\begin{equation}
X_{s}^{t}=I_{n}+%
{\displaystyle\int_{t}^{s}}
A_{r}X_{r}^{t}dr+\sum_{i=1}^{d}%
{\displaystyle\int_{t}^{s}}
D_{r}^{i}X_{r}^{t}dW_{r}^{i}.\label{multi-dynamic-SDE}%
\end{equation}
Denote by $\left(  X_{{}}^{t}\right)  ^{i}$ the $i$th column of $X_{{}}^{t}$
for $i=1,2,\ldots,n$. Obviously,%
\begin{equation}
(X_{s}^{t})^{i}=e_{n}^{i}+%
{\displaystyle\int_{t}^{s}}
A_{r}(X_{r}^{t})^{i}dr+\sum_{j=1}^{d}%
{\displaystyle\int_{t}^{s}}
D_{r}^{j}(X_{r}^{t})^{i}dW_{r}^{j},\label{multi-dynamic-SDE-i}%
\end{equation}
where $e_{n}^{i}$ is the $i$th column of $I_{n}$. Since $X_{s}^{t}$ is
invertible, we have $\left\vert \left(  X_{s}^{t}\right)  ^{i}\right\vert
^{2}>0$ for $s\in\lbrack t,T]$.

Applying It\^{o}'s formula to $\ln\left(  \left\vert \left(  X_{r}^{t}\right)
^{i}\right\vert ^{2}\right)  $ on $[t,s]$, we get%
\[%
\begin{array}
[c]{rl}%
\ln\left(  \left\vert \left(  X_{s}^{t}\right)  ^{i}\right\vert ^{2}\right)
= & 2%
{\displaystyle\int_{t}^{s}}
\frac{\left\langle A_{r}\left(  X_{r}^{t}\right)  ^{i},\left(  X_{r}%
^{t}\right)  ^{i}\right\rangle }{\left\vert \left(  X_{r}^{t}\right)
^{i}\right\vert ^{2}}dr+2\sum\limits_{j=1}^{d}%
{\displaystyle\int_{t}^{s}}
\frac{\left\langle D_{r}^{j}\left(  X_{r}^{t}\right)  ^{i},\left(  X_{r}%
^{t}\right)  ^{i}\right\rangle }{\left\vert \left(  X_{r}^{t}\right)
^{i}\right\vert ^{2}}dW_{r}^{j}\\
& +%
{\displaystyle\int_{t}^{s}}
\sum\limits_{j=1}^{d}\frac{\left\vert D_{r}^{j}\left(  X_{r}^{t}\right)
^{i}\right\vert ^{2}}{\left\vert \left(  X_{r}^{t}\right)  ^{i}\right\vert
^{2}}dr-2%
{\displaystyle\int_{t}^{s}}
\sum\limits_{j=1}^{d}\left(  \frac{\left\langle D_{r}^{j}\left(  X_{r}%
^{t}\right)  ^{i},\left(  X_{r}^{t}\right)  ^{i}\right\rangle }{\left\vert
\left(  X_{r}^{t}\right)  ^{i}\right\vert ^{2}}\right)  ^{2}dr,
\end{array}
\]
which implies that%
\[
\left\vert \left(  X_{s}^{t}\right)  ^{i}\right\vert =\mathcal{E}\left(
\sum\limits_{j=1}^{d}\int_{t}^{s}\phi_{2}^{ij}(r)dW_{r}^{j}\right)
\exp\left\{  \int_{t}^{s}\phi_{1}^{i}(r)dr+\dfrac{1}{2}\int_{t}^{s}%
\sum\limits_{j=1}^{d}\left(  \phi_{3}^{ij}(r)-\left\vert \phi_{2}%
^{ij}(r)\right\vert ^{2}\right)  dr\right\}  ,
\]
where%
\[%
\begin{array}
[c]{ccc}%
\phi_{1}^{i}(r)=\frac{\left\langle A_{r}\left(  X_{r}^{t}\right)  ^{i},\left(
X_{r}^{t}\right)  ^{i}\right\rangle }{\left\vert \left(  X_{r}^{t}\right)
^{i}\right\vert ^{2}}, & \phi_{2}^{ij}(r)=\frac{\left\langle D_{r}^{j}\left(
X_{r}^{t}\right)  ^{i},\left(  X_{r}^{t}\right)  ^{i}\right\rangle
}{\left\vert \left(  X_{r}^{t}\right)  ^{i}\right\vert ^{2}}, & \phi_{3}%
^{ij}(r)=\frac{\left\vert D_{r}^{j}\left(  X_{r}^{t}\right)  ^{i}\right\vert
^{2}}{\left\vert \left(  X_{r}^{t}\right)  ^{i}\right\vert ^{2}}.
\end{array}
\]
By tedious calculation, we obtain
\begin{equation}
\left\Vert \sum\limits_{j=1}^{d}\left(  \phi_{2}^{ij}\cdot W^{j}\right)
\right\Vert _{BMO_{2}}\leq\left(  2\sum\limits_{j=1}^{d}\left(  n\left\Vert
\beta^{j}\cdot W^{j}\right\Vert _{BMO_{2}}^{2}+\left\Vert C_{{}}%
^{j}\right\Vert _{\infty}^{2}T\right)  \right)  ^{\frac{1}{2}}%
\label{Phi-ij-2-W-bound}%
\end{equation}
and%
\[
\left\vert \phi_{1}^{i}(r)\right\vert +\sum\limits_{j=1}^{d}\left\vert
\phi_{3}^{ij}(r)-\left\vert \phi_{2}^{ij}(r)\right\vert ^{2}\right\vert
\leq\left\vert A_{r}\right\vert +C\left(  1+\sum\limits_{j=1}^{d}\left\vert
\beta_{r}^{j}\right\vert \right)  ,
\]
where $C$ is a constant depending on $\sum_{j=1}^{d}\left\Vert C_{{}}%
^{j}\right\Vert _{\infty}$. For the simplicity of writing, we denote
$\mathbb{E}\left[  \cdot\mid\mathcal{F}_{t}\right]  $ by $\mathbb{E}%
_{t}\left[  \cdot\right]  $. Then, for any $p\in\left(  1,\bar{p}\right)  $,
by H\"{o}lder's inequality,
we have%
\begin{equation}%
\begin{array}
[c]{rl}
& \mathbb{E}_{t}\left[  \sup\limits_{s\in\lbrack t,T]}\left\vert \left(
X_{s}^{t}\right)  ^{i}\right\vert ^{p}\right]  \\
\leq & \mathbb{E}_{t}\left[  \sup\limits_{s\in\lbrack t,T]}\mathcal{E}\left(
\sum\limits_{j=1}^{d}%
{\displaystyle\int_{t}^{s}}
\phi_{2}^{ij}(r)dW_{r}^{j}\right)  ^{p}\exp\left\{  pC%
{\displaystyle\int_{t}^{T}}
\left(  1+\left\vert A_{r}\right\vert +\sum\limits_{j=1}^{d}\left\vert
\beta_{r}^{j}\right\vert \right)  dr\right\}  \right]  \\
\leq & \left(  \mathbb{E}_{t}\left[  \sup\limits_{s\in\lbrack t,T]}%
\mathcal{E}\left(  \sum\limits_{j=1}^{d}%
{\displaystyle\int_{t}^{s}}
\phi_{2}^{ij}(r)dW_{r}^{j}\right)  ^{{p}p^{\prime}}\right]  \right)
^{\frac{1}{p^{\prime}}}\\
& \times\left(  \mathbb{E}_{t}\left[  \exp\left\{  pq^{\prime}C%
{\displaystyle\int_{t}^{T}}
\left(  1+\left\vert A_{r}\right\vert +\sum\limits_{j=1}^{d}\left\vert
\beta_{r}^{j}\right\vert \right)  dr\right\}  \right]  \right)  ^{\frac
{1}{q^{^{\prime}}}},
\end{array}
\label{supXts-p}%
\end{equation}
where $p^{\prime}=(p+\bar{p})(2p)^{-1}$ and $q^{\prime}=p^{\prime}(p^{\prime
}-1)^{-1}$.
Thanks to Doob's inequality, reverse H\"{o}lder's inequality and
John-Nirenberg's inequality, we obtain%
\begin{equation}%
\begin{array}
[c]{rl}%
\mathbb{E}_{t}\left[  \sup\limits_{s\in\lbrack t,T]}\mathcal{E}\left(
\sum\limits_{j=1}^{d}%
{\displaystyle\int_{t}^{s}}
\phi_{2}^{ij}(r)dW_{r}^{j}\right)  ^{pp^{\prime}}\right]   & \leq
C\mathbb{E}_{t}\left[  \mathcal{E}\left(  \sum\limits_{j=1}^{d}%
{\displaystyle\int_{t}^{T}}
\phi_{2}^{ij}(r)dW_{r}^{j}\right)  ^{pp^{\prime}}\right]  \\
& \leq CK\left(  pp^{\prime},\left\Vert \sum\limits_{j=1}^{d}\left(  \phi
_{2}^{ij}\cdot W^{j}\right)  \right\Vert _{BMO_{2}}\right)
\end{array}
\label{Et-sup-Gamma-i-Pinte-1}%
\end{equation}
and%
\begin{equation}%
\begin{array}
[c]{l}%
\mathbb{E}_{t}\left[  \exp\left\{  pq^{\prime}C%
{\displaystyle\int_{t}^{T}}
\left(  1+\left\vert A_{r}\right\vert +\sum\limits_{j=1}^{d}\left\vert
\beta_{r}^{j}\right\vert \right)  dr\right\}  \right]  \\
\leq e_{{}}^{pq^{\prime}CT\left(  1+\delta^{-1}\right)  }\mathbb{E}_{t}\left[
\exp\left\{  \frac{pq^{\prime}C\delta}{4}%
{\displaystyle\int_{t}^{T}}
\left(  \left\vert A_{r}\right\vert +\sum\limits_{j=1}^{d}\left\vert \beta
_{r}^{j}\right\vert \right)  ^{2}dr\right\}  \right]  \\
\leq e_{{}}^{pq^{\prime}CT\left(  1+\delta^{-1}\right)  }\left(
1-\frac{pq^{\prime}C\delta}{4}\left\Vert \left(  \left\vert A\right\vert
+\sum\limits_{j=1}^{d}\left\vert \beta^{j}\right\vert \right)  \cdot
W^{1}\right\Vert _{BMO_{2}}^{2}\right)  ^{-1}\\
\leq2e_{{}}^{pq^{\prime}CT\left(  1+\delta^{-1}\right)  },
\end{array}
\label{Et-sup-Gamma-i-Pinte-2}%
\end{equation}
where $\delta=2\left(  pq^{\prime}C\left\Vert \left(  \left\vert A\right\vert
+\sum_{j=1}^{d}\left\vert \beta^{j}\right\vert \right)  \cdot W^{1}\right\Vert
_{BMO_{2}}^{2}\right)  ^{-1}$.
Then, due to the inequalities (\ref{supXts-p})-(\ref{Et-sup-Gamma-i-Pinte-2}), we have $\mathbb{E}_{t}\left[  \sup_{s\in\lbrack t,T]}\left\vert \left(  X_{s}%
^{t}\right)  ^{i}\right\vert ^{p}\right]  \leq C,$
which implies%
\[
\mathbb{E}_{t}\left[  \sup\limits_{s\in\lbrack t,T]}\left\vert X_{s}%
^{t}\right\vert ^{p}\right]  \leq\mathbb{E}_{t}\left[  \left(  \sup
\limits_{s\in\lbrack t,T]}\sum\limits_{i=1}^{n}\left\vert \left(  X_{s}%
^{t}\right)  ^{i}\right\vert ^{2}\right)  ^{\frac{p}{2}}\right]  \leq
C\sum\limits_{i=1}^{n}\mathbb{E}_{t}\left[  \sup\limits_{s\in\lbrack
t,T]}\left\vert \left(  X_{s}^{t}\right)  ^{i}\right\vert ^{p}\right]  \leq C,
\]
where $C$ depends on $p,$ $T,$ $\left\Vert \left(  \left\vert A\right\vert
+\sum_{i=1}^{d}\left\vert \beta^{i}\right\vert \right)  \cdot W^{1}\right\Vert
_{BMO_{2}}$ and $\sum_{i=1}^{d}\left\Vert C_{{}}^{i}\right\Vert _{\infty}$.
\end{proof}

\begin{proposition}
\label{prop-multi-linear-BSDE} Let the assumption in Proposition
\ref{prop-multi-reverseHolder} hold, and $\left(  \xi,f\right)  \in L_{\mathcal{F}_{T}}^{p_{0}^{{}}}(\Omega;\mathbb{R}_{{}}^{n})\times
\mathcal{M}_{\mathcal{F}}^{1,p_{0}^{{}}}([0,T];\mathbb{R}_{{}}^{n})$ for some
$p_{0}^{{}}>\bar{p}^{\ast}$, where $\bar{p}$ is defined in Proposition
\ref{prop-multi-reverseHolder} and $\bar{p}^{\ast}=\bar{p}(\bar{p}-1)^{-1}$.
Then, for any $p\in\left(  \bar{p}^{\ast},p_{0}^{{}}\right)  $, the BSDE
(\ref{multi-linear-BSDE}) admits a unique solution $\left(  Y,Z\right)
\in\mathcal{S}_{\mathcal{F}}^{p}([0,T];\mathbb{R}_{{}}^{n})\times
\mathcal{M}_{\mathcal{F}}^{2,p}([0,T];\mathbb{R}_{{}}^{n\times d})$. Moreover,
we have%
\begin{equation}%
\begin{array}
[c]{rl}%
\mathbb{E}\left[  \sup\limits_{t\in\lbrack0,T]}\left\vert Y_{t}\right\vert
^{p}+\left(  \int_{0}^{T}\left\vert Z_{t}\right\vert ^{2}dt\right)  ^{\frac
{p}{2}}\right]  \leq & C\left(  \mathbb{E}\left[  \left\vert \xi\right\vert
^{p_{0}^{{}}}+\left(  \int_{0}^{T}\left\vert f_{t}\right\vert dt\right)
^{p_{0}^{{}}}\right]  \right)  ^{\frac{p}{p_{0}^{{}}}},
\end{array}
\label{est-multi-BSDE-YZ}%
\end{equation}
where $C$ depends on $p,$ $p_{0}^{{}},$ $T,$ $\left\Vert
\left(  \left\vert A\right\vert +\sum_{i=1}^{d}\left\vert \beta^{i}\right\vert
\right)  \cdot W^{1}\right\Vert _{BMO_{2}}$ and $\sum_{i=1}^{d}\left\Vert
C_{{}}^{i}\right\Vert _{\infty}$.
\end{proposition}

\begin{proof}
We first prove the existence. Set%
\[
\tilde{Y}_{t}:=\mathbb{E}\left[  \Lambda_{t}^{\intercal}X_{T}^{\intercal}%
\xi+\int_{t}^{T}\Lambda_{t}^{\intercal}X_{s}^{\intercal}f_{s}ds\mid
\mathcal{F}_{t}\right]  \text{ for }t\in\lbrack0,T],
\]
where $X$ (resp. $\Lambda$) is the unique strong solution to the SDE
(\ref{multi-linear-SDE}) (resp. (\ref{multi-linear-SDE-inverse})). Obviously,
$\tilde{Y}_{T}=\xi$. For any $p\in\left(  \bar{p}^{\ast},p_{0}^{{}}\right)  $,
by H\"{o}lder's inequality and Proposition \ref{prop-multi-reverseHolder}, we
have%
\begin{equation}%
\begin{array}
[c]{cl}%
\left\vert \tilde{Y}_{t}\right\vert  & \leq\mathbb{E}\left[  \sup
\limits_{s\in\lbrack t,T]}\left\vert \Lambda_{t}^{\intercal}X_{s}^{\intercal
}\right\vert \left(  \left\vert \xi\right\vert +\int_{t}^{T}\left\vert
f_{s}\right\vert ds\right)  \mid\mathcal{F}_{t}\right]  \\
& \leq\left(  \mathbb{E}\left[  \sup\limits_{s\in\lbrack t,T]}\left\vert
X_{s}\Lambda_{t}\right\vert ^{p^{\ast}}\mid\mathcal{F}_{t}\right]  \right)
^{\frac{1}{p^{\ast}}}\left(  \mathbb{E}\left[  \left(  \left\vert
\xi\right\vert +\int_{t}^{T}\left\vert f_{s}\right\vert ds\right)  ^{p}%
\mid\mathcal{F}_{t}\right]  \right)  ^{\frac{1}{p}}\\
& \leq C\left(  \mathbb{E}\left[  \left(  \left\vert \xi\right\vert +\int%
_{0}^{T}\left\vert f_{s}\right\vert ds\right)  ^{p}\mid\mathcal{F}_{t}\right]
\right)  ^{\frac{1}{p}},
\end{array}
\label{multi-Y-absolute}%
\end{equation}
where $p^{\ast}=p(p-1)^{-1}$, and $C$ depends on $p,$ $T,$ $\left\Vert \left(
\left\vert A\right\vert +\sum_{i=1}^{d}\left\vert \beta^{i}\right\vert
\right)  \cdot W^{1}\right\Vert _{BMO_{2}}$ and $\sum_{i=1}^{d}\left\Vert
C_{{}}^{i}\right\Vert _{\infty}$. Thanks to the Doob inequality, we obtain%
\begin{equation}
\left(  \mathbb{E}\left[  \sup_{t\in\lbrack0,T]}\left\vert \tilde{Y}%
_{t}\right\vert ^{p}\right]  \right)  ^{\frac{p_{0}^{{}}}{p}}\leq
C\mathbb{E}\left[  \left\vert \xi\right\vert ^{p_{0}^{{}}}+\left(  \int%
_{0}^{T}\left\vert f_{t}\right\vert dt\right)  ^{p_{0}^{{}}}\right]
,\label{multi-BSDE-est-Y}%
\end{equation}
which implies that $\tilde{Y}\in\mathcal{S}_{\mathcal{F}}^{p}([0,T];\mathbb{R}%
_{{}}^{n})$. Since%
\[
X_{t}^{\intercal}\tilde{Y}_{t}+\int_{0}^{t}X_{s}^{\intercal}f_{s}%
ds=\mathbb{E}\left[  X_{T}^{\intercal}\xi+\int_{0}^{T}X_{s}^{\intercal}%
f_{s}ds\mid\mathcal{F}_{t}\right]  ,
\]
by the martingale representation theorem, we get%
\[
X_{t}^{\intercal}\tilde{Y}_{t}+\int_{0}^{t}X_{s}^{\intercal}f_{s}%
ds=\mathbb{E}\left[  X_{T}^{\intercal}\xi+\int_{0}^{T}X_{s}^{\intercal}%
f_{s}ds\right]  +\sum_{i=1}^{d}\int_{0}^{t}\psi_{s}^{i}dW_{s}^{i},
\]
where $\psi_{{}}^{i}\in\mathcal{M}_{\mathcal{F}}^{2,p}([0,T];\mathbb{R}_{{}%
}^{n})$ for $i=1,2,\ldots,d$. Applying It\^{o}'s formula to $\Lambda
^{\intercal}(X^{\intercal}\tilde{Y})$ and noting that
$d(X_{t}^{\intercal}\tilde{Y}_{t})=-X_{t}^{\intercal}f_{t}dt+\sum_{i=1}%
^{d}\psi_{t}^{i}dW_{t}^{i}$, we get%
\[
\tilde{Y}_{t}=\xi+\int_{t}^{T}\left\{  A_{s}^{\intercal}\tilde{Y}_{s}%
+\sum\limits_{i=1}^{d}\left(  \beta_{s}^{i}I_{n}+C_{s}^{i}\right)
^{\intercal}\tilde{Z}_{s}^{i}+f_{s}\right\}  ds-\sum\limits_{i=1}^{d}\int%
_{t}^{T}\tilde{Z}_{s}^{i}dW_{s}^{i},
\]
where $\tilde{Z}_{t}^{i}=\left(  X_{t}^{\intercal}\right)  ^{-1}\psi_{t}%
^{i}-\left(  \beta_{t}^{i}I_{n}+C_{t}^{i}\right)  ^{\intercal}\tilde{Y}_{t}$
for $t\in\lbrack0,T]$, $i=1,2,\ldots,d$. Set $\tilde{Z}=\left(  \tilde{Z}_{{}%
}^{1},\tilde{Z}_{{}}^{2},\ldots,\tilde{Z}_{{}}^{d}\right)  $ and define
the stopping time
\[
\tau_{m}=\inf\left\{  t\in\lbrack0,T]:\int_{0}^{t}\left\vert \tilde{Z}%
_{s}\right\vert ^{2}ds\geq m\right\}  \wedge T
\]
for each $m\geq1$. Applying It\^{o}'s formula to $\left\vert \tilde{Y}%
_{t}\right\vert ^{2}$ on $[0,\tau_{m}]$, we get%
\[%
\begin{array}
[c]{cl}%
\left\vert \tilde{Y}_{0}\right\vert ^{2}= & \left\vert \tilde{Y}_{\tau_{m}%
}\right\vert ^{2}+2%
{\displaystyle\int_{0}^{\tau_{m}}}
\left\langle A_{t}\tilde{Y}_{t},\tilde{Y}_{t}\right\rangle dt+2%
{\displaystyle\int_{0}^{\tau_{m}}}
\sum\limits_{i=1}^{d}\left\langle D_{t}^{i}\tilde{Y}_{t},\tilde{Z}_{t}%
^{i}\right\rangle dt\\
& +2%
{\displaystyle\int_{0}^{\tau_{m}}}
\left\langle \tilde{Y}_{t},f_{t}\right\rangle dt-2\sum\limits_{i=1}^{d}%
{\displaystyle\int_{0}^{\tau_{m}}}
\left\langle \tilde{Y}_{t},\tilde{Z}_{t}^{i}\right\rangle dW_{t}^{i}-%
{\displaystyle\int_{0}^{\tau_{m}}}
\sum\limits_{i=1}^{d}\left\vert \tilde{Z}_{t}^{i}\right\vert ^{2}dt,
\end{array}
\]
where $D_{t}^{i}=\beta_{t}^{i}I_{n}+C_{t}^{i}$ for $i=1,2,\ldots,d$.
Similar to the proof of the inequality (\ref{est-exp-z}), we obtain%
\[
\mathbb{E}\left[  \left(
{\displaystyle\int_{0}^{\tau_{m}}}
\left\vert \tilde{Z}_{t}\right\vert ^{2}dt\right)  ^{\frac{p}{2}}\right]  \leq
C\left(  \mathbb{E}\left[  \left(  \sup\limits_{t\in\lbrack0,T]}\left\vert
\tilde{Y}_{t}\right\vert ^{p}\right)  ^{\frac{p_{0}^{{}}}{p}}+\left(  \int%
_{0}^{T}\left\vert f_{t}\right\vert dt\right)  ^{p_{0}^{{}}}\right]  \right)
^{\frac{p}{p_{0}^{{}}}},
\]
where $C$ depends on $p,$ $p_{0}^{{}},$ $T,$ $\left\Vert \left(  \left\vert
A\right\vert +\sum_{i=1}^{d}\left\vert \beta^{i}\right\vert \right)  \cdot
W^{1}\right\Vert _{BMO_{2}}$ and $\sum_{i=1}^{d}\left\Vert C_{{}}%
^{i}\right\Vert _{\infty}$. By taking $m\rightarrow\infty$, it follows from
Fatou's lemma that%
\begin{equation}
\mathbb{E}\left[  \left(
{\displaystyle\int_{0}^{T}}
\left\vert \tilde{Z}_{t}\right\vert ^{2}dt\right)  ^{\frac{p}{2}}\right]  \leq
C\left(  \mathbb{E}\left[  \left(  \sup\limits_{t\in\lbrack0,T]}\left\vert
\tilde{Y}_{t}\right\vert ^{p}\right)  ^{\frac{p_{0}^{{}}}{p}}+\left(  \int%
_{0}^{T}\left\vert f_{t}\right\vert dt\right)  ^{p_{0}^{{}}}\right]  \right)
^{\frac{p}{p_{0}^{{}}}}.\label{multi-BSDE-est-Z}%
\end{equation}
Combining (\ref{multi-BSDE-est-Z}) with (\ref{multi-BSDE-est-Y}), we have%
\[
\mathbb{E}\left[  \sup_{t\in\lbrack0,T]}\left\vert \tilde{Y}_{t}\right\vert
^{p}+\left(  \int_{0}^{T}\left\vert \tilde{Z}_{t}\right\vert ^{2}dt\right)
^{\frac{p}{2}}\right]  \leq C\left(  \mathbb{E}\left[  \left\vert
\xi\right\vert ^{p_{0}^{{}}}+\left(  \int_{0}^{T}\left\vert f_{t}\right\vert
dt\right)  ^{p_{0}^{{}}}\right]  \right)  ^{\frac{p}{p_{0}^{{}}}}.
\]

Now we prove the uniqueness. Let $\left(  Y,Z\right)  \in\mathcal{S}%
_{\mathcal{F}}^{p}([0,T];\mathbb{R}_{{}}^{n})\times\mathcal{M}_{\mathcal{F}%
}^{2,p}([0,T];\mathbb{R}_{{}}^{n\times d})$ be a solution to
(\ref{multi-linear-BSDE}). Applying It\^{o}'s formula to $X_{t}^{\intercal}Y$
and taking conditional expectation, we get $Y=\tilde{Y}$. From the
estimate (\ref{multi-BSDE-est-Z}) for $(Y-\tilde{Y},Z-\tilde{Z})$, we obtain
$Z=\tilde{Z}$.
\end{proof}

\subsection{The first and second-order variational equations}

Let $\bar{u}(\cdot)$ be optimal and $(\bar{X}(\cdot),\bar{Y}(\cdot),\bar
{Z}(\cdot))$ be the corresponding state trajectories of (\ref{state-eq}).
Since the control domain is not necessarily convex, we resort to spike
variation method. For any $u(\cdot)\in\mathcal{U}[0,T]$ and any fixed
$t_{0}\in\lbrack0,T)$, define
\[
u_{t}^{\varepsilon}=\left\{
\begin{array}
[c]{lll}%
\bar{u}_{t}, & \ t\in\lbrack0,T]\backslash E_{\varepsilon}, & \\
u_{t}, & \ t\in E_{\varepsilon}, &
\end{array}
\right.
\]
where $E_{\varepsilon}:=[t_{0},t_{0}+\varepsilon]\subset\lbrack0,T]$ with a
sufficiently small $\varepsilon>0$. Let $(X_{{}}^{\varepsilon}(\cdot),Y_{{}%
}^{\varepsilon}(\cdot),Z_{{}}^{\varepsilon}(\cdot))$ be the state trajectories of
(\ref{state-eq}) associated with $u^{\varepsilon}(\cdot)$. For $i=1,2,\ldots
,d$, set%
\[%
\begin{array}
[c]{cc}%
b_{x}(\cdot)=\left(
\begin{array}
[c]{ccc}%
b_{x_{1}}^{1}(\cdot) & \cdots & b_{x_{n}}^{1}(\cdot)\\
\vdots & \ddots & \vdots\\
b_{x_{1}}^{n}(\cdot) & \cdots & b_{x_{n}}^{n}(\cdot)
\end{array}
\right)  , & \sigma_{x}^{i}(\cdot)=\left(
\begin{array}
[c]{ccc}%
\sigma_{x_{1}}^{1i}(\cdot) & \cdots & \sigma_{x_{n}}^{1i}(\cdot)\\
\vdots & \ddots & \vdots\\
\sigma_{x_{1}}^{ni}(\cdot) & \cdots & \sigma_{x_{n}}^{ni}(\cdot)
\end{array}
\right)  .
\end{array}
\]
For simplicity, for $\psi=b$, $\sigma$, $f$, $\Phi$ and $w=x$, $y$, $z$,
denote%
\[%
\begin{array}
[c]{ll}%
\psi(t)=\psi(t,\bar{X}_{t},\bar{Y}_{t},\bar{Z}_{t},\bar{u}_{t}); & \psi
_{w}(t)=\psi_{w}(t,\bar{X}_{t},\bar{Y}_{t},\bar{Z}_{t},\bar{u}_{t});\\
\hat{\psi}(t)=\psi(t,\bar{X}_{t},\bar{Y}_{t},\bar{Z}_{t},u_{t})-\psi(t); &
\hat{\psi}_{w}(t)=\psi_{w}(t,\bar{X}_{t},\bar{Y}_{t},\bar{Z}_{t},u_{t}%
)-\psi_{w}(t).
\end{array}
\]
Moreover, denote by $D^{2}f$ the Hessian matrix of $f$ with respect to $x$,
$y$, $z$.

The first and second-order variational equations for the SDE in (\ref{state-eq})
are
\begin{equation}
\left\{
\begin{array}
[c]{rl}%
dX_{1}(t)= & b_{x}(t)X_{1}(t)dt+%
{\displaystyle\sum\limits_{i=1}^{d}}
\left[  \sigma_{x}^{i}(t)X_{1}(t)+\hat{\sigma}^{i}(t)1_{E_{\varepsilon}%
}(t)\right]  dW_{t}^{i},\\
X_{1}(0)= & 0
\end{array}
\right.  \label{new-form-x1}%
\end{equation}
and%
\begin{equation}
\left\{
\begin{array}
[c]{rl}%
dX_{2}(t)= & \left[  b_{x}(t)X_{2}(t)+\hat{b}(t)1_{E_{\varepsilon}}%
(t)+\dfrac{1}{2}b_{xx}(t)X_{1}(t)X_{1}(t)\right]  dt\\
& +\
{\displaystyle\sum\limits_{i=1}^{d}}
\left[  \sigma_{x}^{i}(t)X_{2}(t)+\hat{\sigma}_{x}^{i}(t)X_{1}%
(t)1_{E_{\varepsilon}}(t)+\dfrac{1}{2}\sigma_{xx}^{i}(t)X_{1}(t)X_{1}%
(t)\right]  dW_{t}^{i},\\
X_{2}(0)= & 0
\end{array}
\right.  \label{new-form-x2}%
\end{equation}
respectively, where
\[
\sigma_{xx}^{i}(t)X_{1}(t)X_{1}(t):=\left(  \mathrm{tr}\left\{  \sigma
_{xx}^{1i}(t)X_{1}(t)X_{1}^{\intercal}(t)\right\}  ,\ldots,\mathrm{tr}\left\{
\sigma_{xx}^{ni}(t)X_{1}(t)X_{1}^{\intercal}(t)\right\}  \right)  ^{\intercal}%
\]
for $i=1,2,\ldots,d$, and similarly for $b_{xx}(t)X_{1}(t)X_{1}(t)$. The
following results can be found in \cite{Peng90} (see also \cite{YongZhou}).

\begin{lemma}
\label{est-forward-expansion-lem} Suppose (i) and (ii) in Assumption
\ref{assum-2} hold. Then, for any $\beta>1$, (\ref{new-form-x1}) (resp.
(\ref{new-form-x2})) has a unique solution $X_{1}(\cdot)\in$ $\mathcal{S}%
_{\mathcal{F}}^{\beta}([0,T];\mathbb{R}^{n})$ (resp. $X_{2}(\cdot)\in$
$\mathcal{S}_{\mathcal{F}}^{\beta}([0,T];\mathbb{R}^{n})$). Moreover,%
\begin{equation}
\mathbb{E}\left[  \sup\limits_{t\in\lbrack0,T]}\left\vert X_{t}^{\varepsilon
}-\bar{X}_{t}\right\vert ^{\beta}\right]  =O\left(  \varepsilon^{\frac{\beta
}{2}}\right)  , \label{est-forward-expansion-1}%
\end{equation}%
\begin{equation}
\mathbb{E}\left[  \sup\limits_{t\in\lbrack0,T]}\left\vert X_{1}(t)\right\vert
^{\beta}\right]  =O\left(  \varepsilon^{\frac{\beta}{2}}\right)  ,
\label{est-forward-expansion-2}%
\end{equation}%
\begin{equation}
\mathbb{E}\left[  \sup\limits_{t\in\lbrack0,T]}\left\vert X_{t}^{\varepsilon
}-\bar{X}_{t}-X_{1}(t)\right\vert ^{\beta}\right]  =O\left(  \varepsilon
^{\beta}\right)  , \label{est-forward-expansion-3}%
\end{equation}%
\begin{equation}
\mathbb{E}\left[  \sup\limits_{t\in\lbrack0,T]}\left\vert X_{2}(t)\right\vert
^{\beta}\right]  =O\left(  \varepsilon^{\beta}\right)  ,
\label{est-forward-expansion-4}%
\end{equation}%
\begin{equation}
\mathbb{E}\left[  \sup\limits_{t\in\lbrack0,T]}\left\vert X_{t}^{\varepsilon
}-\bar{X}_{t}-X_{1}(t)-X_{2}(t)\right\vert ^{\beta}\right]  =o\left(
\varepsilon^{\beta}\right)  . \label{est-forward-expansion-5}%
\end{equation}

\end{lemma}

Set $\left(  \xi^{1,\varepsilon}(t),\eta^{1,\varepsilon}(t),\zeta
^{1,\varepsilon}(t)\right)  :=\left(  X_{t}^{\varepsilon}-\bar{X}_{t}%
,Y_{t}^{\varepsilon}-\bar{Y}_{t},Z_{t}^{\varepsilon}-\bar{Z}_{t}\right)
;\Theta_{t}:=(\bar{X}_{t},\bar{Y}_{t},\bar{Z}_{t}); $\\
$ \Theta_{t}^{\varepsilon
}:=(X_{t}^{\varepsilon},Y_{t}^{\varepsilon},Z_{t}^{\varepsilon}).$ We have%
\begin{equation}
\left\{
\begin{array}
[c]{rl}%
d\eta^{1,\varepsilon}(t)= & -\left[  \left(  \tilde{f}_{x}^{\varepsilon
}(t)\right)  ^{\intercal}\xi^{1,\varepsilon}(t)+\tilde{f}_{y}^{\varepsilon
}(t)\eta^{1,\varepsilon}(t)+\left(  \tilde{f}_{z}^{\varepsilon}(t)\right)
^{\intercal}\zeta^{1,\varepsilon}(t)+\hat{f}(t)1_{E_{\varepsilon}}(t)\right]
dt\\
& +\left(  \zeta^{1,\varepsilon}(t)\right)  ^{\intercal}dW_{t},\\
\eta^{1,\varepsilon}(T)= & \left(  \tilde{\Phi}_{x}^{\varepsilon}(T)\right)
^{\intercal}\xi^{1,\varepsilon}(T),
\end{array}
\right.  \label{ep-bar-y}%
\end{equation}
where $\tilde{f}_{z}^{\varepsilon}(t)=\int_{0}^{1}f_{z}(t,\Theta_{t}%
+\theta(\Theta_{t}^{\varepsilon}-\Theta_{t}),u_{t}^{\varepsilon})d\theta$.
$\tilde{f}_{x}^{\varepsilon}(t)$, $\tilde{f}_{y}^{\varepsilon}(t)$ and
$\tilde{\Phi}_{x}^{\varepsilon}(T)$ are defined similarly. The following lemma
is the estimate for $\left(  \eta^{1,\varepsilon}(\cdot),\zeta^{1,\varepsilon
}(\cdot)\right)  $.

\begin{lemma}
\label{est-epsilon-bar} Suppose that Assumption \ref{assum-2} holds. Then we have
\begin{equation}%
\begin{array}
[c]{rl}%
\mathbb{E}\left[  \sup\limits_{t\in\lbrack0,T]}\left\vert \eta^{1,\varepsilon
}(t)\right\vert ^{\beta}+\left(  \int_{0}^{T}\left\vert \zeta^{1,\varepsilon
}(t)\right\vert ^{2}dt\right)  ^{\frac{\beta}{2}}\right]  =O\left(
\varepsilon^{\frac{\beta}{2}}\right)  \text{ } & \text{for any }\beta
\in(1,+\infty).
\end{array}
\label{est-ep-yz}%
\end{equation}

\end{lemma}

\begin{proof}
For any $\beta>p_{\tilde{f}_{z}^{\varepsilon}}^{\ast}$, set $\beta_{0}%
=\beta+1$, it follows from Proposition \ref{est-linear-prop}, the estimate
(\ref{est-forward-expansion-1}) and the energy inequality for $\bar{Z}\cdot W$
that
\[%
\begin{array}
[c]{l}%
\mathbb{E}\left[  \sup\limits_{t\in\lbrack0,T]}\left\vert \eta^{1,\varepsilon
}(t)\right\vert ^{\beta}+\left(  \int_{0}^{T}\left\vert \zeta^{1,\varepsilon
}(t)\right\vert ^{2}dt\right)  ^{\frac{\beta}{2}}\right]  \\
\leq C\left(  \mathbb{E}\left[  \left\vert \left(  \tilde{\Phi}_{x}%
^{\varepsilon}(T)\right)  ^{\intercal}\xi^{1,\varepsilon}(T)\right\vert
^{\beta_{0}}+\left(  \int_{0}^{T}\left\vert \left(  \tilde{f}_{x}%
^{\varepsilon}(t)\right)  ^{\intercal}\xi^{1,\varepsilon}(t)+\hat
{f}(t)1_{E_{\varepsilon}}(t)\right\vert dt\right)  ^{\beta_{0}}\right]
\right)  ^{\frac{\beta}{\beta_{0}}}\\
\leq C\left\{  \left(  \mathbb{E}\left[  \sup\limits_{t\in\lbrack
0,T]}\left\vert \xi^{1,\varepsilon}(t)\right\vert ^{\beta_{0}}\right]
\right)  ^{\frac{\beta}{\beta_{0}}}+\left(  \mathbb{E}\left[  \left(
\int_{E_{\varepsilon}}\left\vert \hat{f}(t)\right\vert dt\right)  ^{\beta_{0}%
}\right]  \right)  ^{\frac{\beta}{\beta_{0}}}\right\}  \\
\leq C\varepsilon^{\frac{\beta}{2}}\left\{  1+\left(  \mathbb{E}\left[
\left(  \int_{0}^{T}\left(  1+\left\vert \bar{Z}_{t}\right\vert \right)
^{2}dt\right)  ^{\frac{\beta_{0}}{2}}\right]  \right)  ^{\frac{\beta}%
{\beta_{0}}}\right\}  \\
\leq C\varepsilon^{\frac{\beta}{2}},
\end{array}
\]
where $C$ depends on $\beta,$ $\left\Vert \Phi\right\Vert _{\infty},$
$\left\Vert f_{x}\right\Vert _{\infty},$ $\left\Vert f_{y}\right\Vert
_{\infty},$ $T$ and $\left\Vert \tilde{f}_{z}^{\varepsilon}\cdot W\right\Vert
_{BMO_{2}}$. Since $\left\vert \tilde{f}_{z}^{\varepsilon}(t)\right\vert \leq
L_{3}+\gamma\left(  \left\vert \bar{Z}_{t}\right\vert +\left\vert
Z_{t}^{\varepsilon}\right\vert \right)  $, by the second inequality in
(\ref{est-quad-qfbsde}), $C$ depends on $\beta,$ $\left\Vert \Phi\right\Vert
_{\infty},$ $\left\Vert f_{x}\right\Vert _{\infty},$ $\left\Vert
f_{y}\right\Vert _{\infty},$ $T,$ $\alpha,$ $L_{3}$ and $\gamma$. The case
where $\beta\in\left(  1,p_{\tilde{f}_{z}^{\varepsilon}}^{\ast}\right]  $
follows immediately from H\"{o}lder's inequality.
\end{proof}

In order to obtain the first-order variational equation of the BSDE in
(\ref{state-eq}), we introduce the first-order adjoint equation%
\begin{equation}%
\begin{array}
[c]{cc}%
p_{t}= & \Phi_{x}(\bar{X}_{T})+%
{\displaystyle\int_{t}^{T}}
\left\{  \left[  \sum\limits_{i=1}^{d}f_{z_{i}}(s)\left(  \sigma_{x}%
^{i}(s)\right)  ^{\intercal}+f_{y}(s)I_{n}+b_{x}^{\intercal}(s)\right]
p_{s}\right. \\
& +\left.  \sum\limits_{i=1}^{d}\left[  f_{z_{i}}(s)I_{n}+\left(  \sigma
_{x}^{i}(s)\right)  ^{\intercal}\right]  q_{s}^{i}+f_{x}(s)\right\}
ds-\sum\limits_{i=1}^{d}%
{\displaystyle\int_{t}^{T}}
q_{s}^{i}dW_{s}^{i}.
\end{array}
\label{eq-p}%
\end{equation}

\begin{lemma}
\label{adj-1st-lem} Suppose that Assumption \ref{assum-2} holds. Then there exists a unique solution
$\left(  p(\cdot),q(\cdot)\right)  \in$ \\
$L_{\mathcal{F}}^{\infty}([0,T];\mathbb{R}_{{}}^{n})\times\bigcap_{\beta
>1}\mathcal{M}_{\mathcal{F}}^{2,\beta}([0,T];\mathbb{R}_{{}}^{n\times d})$ to the equation (\ref{eq-p}),
where $q(\cdot)=\left(  q_{{}}^{1}(\cdot),q_{{}}^{2}(\cdot),\ldots,q_{{}}%
^{d}(\cdot)\right)  $.
\end{lemma}

\begin{proof}
Since $\left(  \sum_{i=1}^{d}|f_{z_{i}}|\right)  \cdot W^{1}\in BMO$, and
$b_{x},$ $\left(  \sigma_{x}^{1},\ldots,\sigma_{x}^{d}\right)  ,$ $f_{x},$
$f_{y}$ and $\Phi_{x}$ are essentially bounded, by Proposition \ref{prop-multi-linear-BSDE}, (\ref{eq-p}) has
a unique solution $\left(  p(\cdot),q(\cdot)\right)  \in\bigcap_{\beta
>1}\left(  \mathcal{S}_{\mathcal{F}}^{\beta}([0,T];\mathbb{R}_{{}}^{n}%
)\times\mathcal{M}_{\mathcal{F}}^{2,\beta}([0,T];\mathbb{R}_{{}}^{n\times
d})\right)  $. By the inequality
(\ref{multi-Y-absolute}) in the proof of Proposition
\ref{prop-multi-linear-BSDE}, we can easily deduce $p(\cdot)\in L_{\mathcal{F}%
}^{\infty}([0,T];\mathbb{R}_{{}}^{n})$.
\end{proof}

Now we introduce the first-order variational equation for the backward state equation:
\begin{equation}
\left\{
\begin{array}
[c]{ll}%
dY_{1}(t)= & -\left\{  \left(  f_{x}(t)\right)  ^{\intercal}X_{1}%
(t)+f_{y}(t)Y_{1}(t)+\left(  f_{z}(t)\right)  ^{\intercal}Z_{1}(t)\right.  \\
& -\left.  \left[  \left(  f_{z}(t)\right)  ^{\intercal}\left(  \hat{\sigma
}(t)\right)  ^{\intercal}p_{t}+\sum\limits_{i=1}^{d}\left(  \hat{\sigma}%
^{i}(t)\right)  ^{\intercal}q_{t}^{i}\right]  1_{E_{\varepsilon}}(t)\right\}
dt+\left(  Z_{1}(t)\right)  ^{\intercal}dW_{t},\\
Y_{1}(T)= & \left(  \Phi_{x}(\bar{X}_{T})\right)  ^{\intercal}X_{1}(T),
\end{array}
\right.  \label{new-form-y1}%
\end{equation}
where $X_{1}$ is defined in (\ref{new-form-x1}).

\begin{remark}
One can observe that the variational equations (\ref{new-form-y1}), (\ref{new-form-y2}) and adjoint equations (\ref{eq-p}), (\ref{eq-P}) are the
same as those in \cite{Hu17}, although the well-posedness of them is different from
that in \cite{Hu17}. In fact, no matter what assumptions are imposed on the coefficients of the control system (\ref{state-eq}),
we can first assume that the expressions of the variational equation and the adjoint equation are the same as
those in \cite{Hu17}, and then give rigorous proofs of their well-posedness and corresponding estimates.
For the fully coupled case considered in \cite{HuJiXue18}, this approach is summarized as a heuristic derivation.
\end{remark}

The following well-posedness result for $(Y_{1}(\cdot),Z_{1}%
(\cdot))$\ is a direct application of Proposition \ref{est-linear-prop}.

\begin{lemma}
\label{exist-y1-lem} Suppose that Assumption \ref{assum-2} holds. Then there
exists a unique solution $(Y_{1}(\cdot),Z_{1}(\cdot))\in\bigcap_{\beta>1}\left(  \mathcal{S}%
_{\mathcal{F}}^{\beta}([0,T];\mathbb{R})\times\mathcal{M}_{\mathcal{F}%
}^{2,\beta}([0,T];\mathbb{R}^{d})\right)$ to (\ref{new-form-y1}).
\end{lemma}

The relationship between $(Y_{1}(\cdot),Z_{1}(\cdot))$ and $X_{1}(\cdot)$ is
obtained in the following lemma.

\begin{lemma}
\label{relation-y1z1} Suppose that Assumption \ref{assum-2} holds. Then we
have%
\begin{equation}%
\begin{array}
[c]{l}%
Y_{1}(t)=p_{t}^{\intercal}X_{1}(t),\\
Z_{1}^{i}(t)=p_{t}^{\intercal}\hat{\sigma}^{i}(t)1_{E_{\varepsilon}%
}(t)+\left[  p_{t}^{\intercal}\sigma_{x}^{i}(t)+\left(  q_{t}^{i}\right)
^{\intercal}\right]  X_{1}(t),\text{ }i=1,2,\ldots,d,
\end{array}
\label{decoupled-relation-y1-z1}%
\end{equation}
where $\left(  p(\cdot),q(\cdot)\right)  $ is the solution to the equation
(\ref{eq-p}).
\end{lemma}

\begin{proof}
Applying It\^{o}'s formula to $p_{t}^{\intercal}X_{1}(t)$, we can obtain the
desired result by the uniqueness of the solution to (\ref{new-form-y1}).
\end{proof}

\begin{lemma}
\label{est-one-order} Suppose that Assumption \ref{assum-2} holds. Then, for
any $\beta>1$, we have
\begin{equation}
\mathbb{E}\left[  \sup\limits_{t\in\lbrack0,T]}\left\vert Y_{1}(t)\right\vert
^{\beta}+\left(  \int_{0}^{T}\left\vert Z_{1}(t)\right\vert ^{2}dt\right)
^{\frac{\beta}{2}}\right]  =O\left(  \varepsilon^{\frac{\beta}{2}}\right)
,\label{est-x1-y1}%
\end{equation}
and for any $\beta\in\left(  1\vee2p_{f_{z}}^{-1},2\right)  $,%
\begin{equation}
\mathbb{E}\left[  \sup\limits_{t\in\lbrack0,T]}\Gamma_{t}\left\vert
Y_{t}^{\varepsilon}-\bar{Y}_{t}-Y_{1}(t)\right\vert ^{\beta}+\left(  \int%
_{0}^{T}\left(  \Gamma_{t}\right)  ^{\frac{2}{\beta}}\left\vert Z_{t}%
^{\varepsilon}-\bar{Z}_{t}-Z_{1}(t)\right\vert ^{2}dt\right)  ^{\frac{\beta
}{2}}\right]  =O\left(  \varepsilon^{\beta}\right)  ,\label{est-eta2-zeta2}%
\end{equation}
where $\Gamma_{t}:=\mathcal{E}\left(  \int_{0}^{t}\left(  f_{z}(s)\right)
^{\intercal}dW_{s}\right)  $ for $t\in\lbrack0,T]$.
\end{lemma}

\begin{proof}
For simplicity, we only prove the case $d=1$, and the proof for the case $d>1$
is simliar. By the estimate (\ref{est-forward-expansion-2}), Lemma
\ref{adj-1st-lem} and H\"{o}lder's inequality, for $1<\beta<\beta_{0}$, we get
\begin{equation}%
\begin{array}
[c]{rl}
& \mathbb{E}\left[  \left(  \int_{0}^{T}\left\vert q_{t}^{\intercal}%
X_{1}(t)\right\vert ^{2}dt\right)  ^{\frac{\beta}{2}}\right]  \\
\leq & \mathbb{E}\left[  \sup\limits_{t\in\lbrack0,T]}\left\vert
X_{1}(t)\right\vert ^{\beta}\left(  \int_{0}^{T}\left\vert q_{t}\right\vert
^{2}dt\right)  ^{\frac{\beta}{2}}\right]  \\
\leq & \left(  \mathbb{E}\left[  \sup\limits_{t\in\lbrack0,T]}\left\vert
X_{1}(t)\right\vert ^{\beta_{0}}\right]  \right)  ^{\frac{\beta}{\beta_{0}}%
}\left(  \mathbb{E}\left[  \left(  \int_{0}^{T}\left\vert q_{t}\right\vert
^{2}dt\right)  ^{\frac{\beta\beta_{0}}{2(\beta_{0}-\beta)}}\right]  \right)
^{\frac{\beta_{0}-\beta}{\beta_{0}}}\\
\leq & C\varepsilon^{\frac{\beta}{2}}.
\end{array}
\label{est-1st-expan-qX1}%
\end{equation}
The other terms in (\ref{decoupled-relation-y1-z1}) can be estimated like
(\ref{est-1st-expan-qX1}). Therefore, we obtain (\ref{est-x1-y1}).

We use the notations $\left(  \xi^{1,\varepsilon}(t),\eta^{1,\varepsilon
}(t),\zeta^{1,\varepsilon}(t)\right)  $, $\tilde{f}_{x}^{\varepsilon}(t)$,
$\tilde{f}_{y}^{\varepsilon}(t)$, $\tilde{f}_{z}^{\varepsilon}(t)$ and
$\tilde{\Phi}_{x}^{\varepsilon}(T)$ in the proof of Lemma
\ref{est-epsilon-bar}. Set $$\left(  \xi^{2,\varepsilon}(t),\eta^{2,\varepsilon
}(t),\zeta^{2,\varepsilon}(t)\right)  :=\left(  \xi^{1,\varepsilon}%
(t)-X_{1}(t),\eta^{1,\varepsilon}(t)-Y_{1}(t),\zeta^{1,\varepsilon}%
(t)-Z_{1}(t)\right),  $$ we have%
\begin{equation}
\left\{
\begin{array}
[c]{rl}%
d\eta^{2,\varepsilon}(t)= & -\left[  \left(  f_{x}(t)\right)  ^{\intercal}%
\xi^{2,\varepsilon}(t)+f_{y}(t)\eta^{2,\varepsilon}(t)+f_{z}(t)\zeta
^{2,\varepsilon}(t)+A_{1}^{\varepsilon}(t)\right]  dt\\
& +\zeta^{2,\varepsilon}(t)dW_{t},\\
\eta^{2,\varepsilon}(T)= & \left(  \tilde{\Phi}_{x}^{\varepsilon}(T)\right)
^{\intercal}\xi^{2,\varepsilon}(T)+B_{1}^{\varepsilon}(T),
\end{array}
\right.  \label{eta-eps-1}%
\end{equation}
where%
\[%
\begin{array}
[c]{rl}%
A_{1}^{\varepsilon}(t)= & \left[  \tilde{f}_{x}^{\varepsilon}(t)-f_{x}%
(t)\right]  ^{\intercal}\xi^{1,\varepsilon}(t)+\left[  \tilde{f}%
_{y}^{\varepsilon}(t)-f_{y}(t)\right]  \eta^{1,\varepsilon}(t)+\left[
\tilde{f}_{z}^{\varepsilon}(t)-f_{z}(t)\right]  \zeta^{1,\varepsilon}(t)\\
& +\ \hat{f}(t)1_{E_{\varepsilon}}(t)+\left(  \hat{\sigma}(t)\right)
^{\intercal}\left[  f_{z}(t)p_{t}+q_{t}\right]  1_{E_{\varepsilon}}(t),\\
B_{1}^{\varepsilon}(T)= & \left[  \tilde{\Phi}_{x}^{\varepsilon}(T)-\Phi
_{x}(\bar{X}_{T})\right]  ^{\intercal}X_{1}(T).
\end{array}
\]
For any $\beta\in\left(  1\vee2p_{f_{z}}^{-1},2\right)  $ and $\beta_{0}%
\in(\beta,2)$, by Proposition \ref{est-exp-yz-prop}, we have
\begin{equation}%
\begin{array}
[c]{l}%
\mathbb{E}\left[  \sup\limits_{t\in\lbrack0,T]}\Gamma_{t}\left\vert
\eta^{2,\varepsilon}(t)\right\vert ^{\beta}\right]  +\mathbb{E}\left[  \left(
\int_{0}^{T}\left(  \Gamma_{t}\right)  ^{\frac{2}{\beta}}\left\vert
\zeta^{2,\varepsilon}(t)\right\vert ^{2}dt\right)  ^{\frac{\beta}{2}}\right]
\\
\leq C\left(  \mathbb{E}\left[  \left(  \Gamma_{T}\right)  ^{\frac{\beta_{0}%
}{\beta}}\left\vert \eta^{2,\varepsilon}(T)\right\vert ^{\beta_{0}}+\left(
\int_{0}^{T}\left(  \Gamma_{t}\right)  ^{\frac{1}{\beta}}\left\vert \left(
f_{x}(t)\right)  ^{\intercal}\xi^{2,\varepsilon}(t)+A_{1}^{\varepsilon
}(t)\right\vert dt\right)  ^{\beta_{0}}\right]  \right)  ^{\frac{\beta}%
{\beta_{0}}}.
\end{array}
\label{est-eta2-zeta2-pre}%
\end{equation}
Due to the definition of $\tilde{f}_{z}^{\varepsilon}(t)$ and Assumption 2.2 (iii), we get
\[%
\begin{array}
[c]{rl}%
\left\vert \tilde{f}_{z}^{\varepsilon}(t)-f_{z}(t)\right\vert \leq &
\left\vert \tilde{f}_{z}^{\varepsilon}(t)-f_{z}(t,\bar{X}_{t},\bar{Y}_{t}%
,\bar{Z}_{t},u_{t}^{\varepsilon})\right\vert +\left\vert \hat{f}%
_{z}(t)\right\vert 1_{E_{\varepsilon}}(t)\\
\leq & C\left[  \left\vert \xi^{1,\varepsilon}(t)\right\vert +\left\vert
\eta^{1,\varepsilon}(t)\right\vert +\left\vert \zeta^{1,\varepsilon
}(t)\right\vert +(1+\left\vert \bar{Z}_{t}\right\vert )1_{E_{\varepsilon}%
}(t)\right]  .
\end{array}
\]
We only estimate the most important and difficult terms in
(\ref{est-eta2-zeta2-pre}) as follows.

By virtue of Assumption \ref{assum-2}, $\left\vert \tilde{\Phi}_{x}^{\varepsilon}(T)-\Phi_{x}%
(\bar{X}_{T})\right\vert \leq C\left\vert \xi^{1,\varepsilon}(T)\right\vert $.
Then, by H\"{o}lder's inequality, reverse H\"{o}lder's inequality, the estimates (\ref{est-forward-expansion-1}%
) and (\ref{est-x1-y1}),
\[%
\begin{array}
[c]{l}%
\mathbb{E}\left[  \left(  \Gamma_{T}\right)  ^{\frac{\beta_{0}}{\beta}%
}\left\vert B_{1}^{\varepsilon}(T)\right\vert ^{\beta_{0}}\right]  \\
\leq C\left(  \mathbb{E}\left[  \left(  \Gamma_{T}\right)  ^{\frac{2}{\beta}%
}\right]  \right)  ^{\frac{\beta_{0}}{2}}\left(  \mathbb{E}\left[
\sup\limits_{t\in\lbrack0,T]}\left\vert X_{1}(t)\right\vert ^{\frac{4\beta
_{0}}{2-\beta_{0}}}\right]  \right)  ^{\frac{2-\beta_{0}}{4}}\left(
\mathbb{E}\left[  \sup\limits_{t\in\lbrack0,T]}\left\vert \xi^{1,\varepsilon
}(t)\right\vert ^{\frac{4\beta_{0}}{2-\beta_{0}}}\right]  \right)
^{\frac{2-\beta_{0}}{4}}\\
\leq C\varepsilon^{\beta_{0}}.
\end{array}
\]

Due to H\"{o}lder's inequality, we have%
\begin{equation}%
\begin{array}
[c]{l}%
\mathbb{E}\left[  \left(  \int_{E_{\varepsilon}}\left(  \Gamma_{t}\right)
^{\frac{1}{\beta}}\left\vert \bar{Z}_{t}\right\vert \left\vert \zeta
^{1,\varepsilon}(t)\right\vert dt\right)  ^{\beta_{0}}\right]  \\
\leq\mathbb{E}\left[  \left(  \int_{E_{\varepsilon}}\left(  \Gamma_{t}\right)
^{\frac{2}{\beta}}\left\vert \bar{Z}_{t}\right\vert ^{2}dt\right)
^{\frac{\beta_{0}}{2}}\left(  \int_{0}^{T}\left\vert \zeta^{1,\varepsilon
}(t)\right\vert ^{2}dt\right)  ^{\frac{\beta_{0}}{2}}\right]  \\
\leq\left(  \mathbb{E}\left[  \int_{E_{\varepsilon}}\left(  \Gamma_{t}\right)
^{\frac{2}{\beta}}\left\vert \bar{Z}_{t}\right\vert ^{2}dt\right]  \right)
^{\frac{\beta_{0}}{2}}\left(  \mathbb{E}\left[  \left(  \int_{0}^{T}\left\vert
\zeta^{1,\varepsilon}(t)\right\vert ^{2}dt\right)  ^{\frac{\beta_{0}}%
{2-\beta_{0}}}\right]  \right)  ^{\frac{2-\beta_{0}}{2}}.
\end{array}
\label{est-exp-Zbar-zeta1}%
\end{equation}
Thanks to H\"{o}lder's inequality, Doob's inequality and reverse H\"{o}lder's
inequality, for any fixed \\
$\beta_{1}\in\left(  1\vee2p_{f_{z}}^{-1},2\right)
$ such that $\beta_{1}<\beta$, we get%
\[%
\begin{array}
[c]{rl}
& \mathbb{E}\left[  \int_{0}^{T}\left(  \Gamma_{t}\right)  ^{\frac{2}{\beta}%
}\left\vert \bar{Z}_{t}\right\vert ^{2}dt\right]  \\
\leq & \left(  \mathbb{E}\left[  \sup\limits_{t\in\lbrack0,T]}\left(
\Gamma_{t}\right)  ^{\frac{2}{\beta_{1}}}\right]  \right)  ^{\frac{\beta_{1}%
}{\beta}}\left(  \mathbb{E}\left[  \left(  \int_{0}^{T}\left\vert \bar{Z}%
_{t}\right\vert ^{2}dt\right)  ^{\frac{\beta}{\beta-\beta_{1}}}\right]
\right)  ^{\frac{\beta-\beta_{1}}{\beta}}\\
\leq & C\left(  \mathbb{E}\left[  \left(  \Gamma_{T}\right)  ^{\frac{2}%
{\beta_{1}}}\right]  \right)  ^{\frac{\beta_{1}}{\beta}}\left(  \mathbb{E}%
\left[  \left(  \int_{0}^{T}\left\vert \bar{Z}_{t}\right\vert ^{2}dt\right)
^{\frac{\beta}{\beta-\beta_{1}}}\right]  \right)  ^{\frac{\beta-\beta_{1}%
}{\beta}}\\
< & \infty,
\end{array}
\]
which implies%
\begin{equation}
\mathbb{E}\left[
{\displaystyle\int_{E_{\varepsilon}}}
\left(  \Gamma_{t}\right)  ^{\frac{2}{\beta}}\left\vert \bar{Z}_{t}\right\vert
^{2}dt\right]  \leq C\varepsilon.\label{est-exp-fzeps-fz-zeta1}%
\end{equation}
It follows from Lemma \ref{est-epsilon-bar}, the inequalities
(\ref{est-exp-Zbar-zeta1}) and (\ref{est-exp-fzeps-fz-zeta1}) that%
\[
\mathbb{E}\left[  \left(
{\displaystyle\int_{E_{\varepsilon}}}
\left(  \Gamma_{t}\right)  ^{\frac{1}{\beta}}\left\vert \bar{Z}_{t}\right\vert
\left\vert \zeta^{1,\varepsilon}(t)\right\vert dt\right)  ^{\beta_{0}}\right]
\leq C\varepsilon^{\beta_{0}}.
\]

From the estimate (\ref{est-exp-fzeps-fz-zeta1}), we have%
\[
\mathbb{E}\left[  \left(
{\displaystyle\int_{E_{\varepsilon}}}
\left(  \Gamma_{t}\right)  ^{\frac{1}{\beta}}|\hat{f}(t)|dt\right)
^{\beta_{0}}\right]  \leq C\left(  \mathbb{E}\left[
{\displaystyle\int_{E_{\varepsilon}}}
\left(  \Gamma_{t}\right)  ^{\frac{2}{\beta}}\left(  1+\left\vert \bar{Z}%
_{t}\right\vert ^{2}\right)  dt\right]  \right)  ^{\frac{\beta_{0}}{2}%
}\varepsilon^{\frac{\beta_{0}}{2}}\leq C\varepsilon^{\beta_{0}}%
\]
and%
\[%
\begin{array}
[c]{l}%
\mathbb{E}\left[  \left(  \int_{E_{\varepsilon}}\left(  \Gamma_{t}\right)
^{\frac{1}{\beta}}\left\vert \bar{Z}_{t}\right\vert \left\vert u_{t}%
\right\vert dt\right)  ^{\beta_{0}}\right]  \\
\leq\mathbb{E}\left[  \left(  \int_{E_{\varepsilon}}\left(  \Gamma_{t}\right)
^{\frac{2}{\beta}}\left\vert \bar{Z}_{t}\right\vert ^{2}dt\right)
^{\frac{\beta_{0}}{2}}\left(  \int_{E_{\varepsilon}}\left\vert u_{t}%
\right\vert ^{2}dt\right)  ^{\frac{\beta_{0}}{2}}\right]  \\
\leq\left(  \mathbb{E}\left[  \int_{E_{\varepsilon}}\left(  \Gamma_{t}\right)
^{\frac{2}{\beta}}\left\vert \bar{Z}_{t}\right\vert ^{2}dt\right]  \right)
^{\frac{\beta_{0}}{2}}\left(  \mathbb{E}\left[  \left(  \int_{E_{\varepsilon}%
}\left\vert u_{t}\right\vert ^{2}dt\right)  ^{\frac{\beta_{0}}{2-\beta_{0}}%
}\right]  \right)  ^{\frac{2-\beta_{0}}{2}}\\
\leq\left(  \mathbb{E}\left[  \int_{E_{\varepsilon}}\left(  \Gamma_{t}\right)
^{\frac{2}{\beta}}\left\vert \bar{Z}_{t}\right\vert ^{2}dt\right]  \right)
^{\frac{\beta_{0}}{2}}\left(  \mathbb{E}\left[  \int_{E_{\varepsilon}%
}\left\vert u_{t}\right\vert ^{\frac{2\beta_{0}}{2-\beta_{0}}}dt^{{}}\right]
\right)  ^{\frac{2-\beta_{0}}{2}}\varepsilon^{\beta_{0}-1}\\
\leq C\varepsilon^{\beta_{0}}.
\end{array}
\]
Hence, we finally obtain (\ref{est-eta2-zeta2}).
\end{proof}

\bigskip

Now define $\Delta(t):=\left(  \Delta^{1}(t),\Delta^{2}(t),\ldots,\Delta
^{d}(t)\right)  ^{\intercal}$, where $\Delta^{i}(t):=\left(  \hat{\sigma}%
^{i}(t)\right)  ^{\intercal}p_{t}$ for $i=1,2,\ldots,d$ and $t\in\lbrack0,T]$.
For $w=x$, $y$, $z$, denote%
\begin{equation}%
\begin{array}
[c]{l}%
\hat{f}(t,\Delta)=f(t,\bar{X}_{t},\bar{Y}_{t},\bar{Z}_{t}+\Delta
(t),u_{t})-f(t),\\
\hat{f}_{w}(t,\Delta)=f_{w}(t,\bar{X}_{t},\bar{Y}_{t},\bar{Z}_{t}%
+\Delta(t),u_{t})-f_{w}(t),
\end{array}
\label{f-fw-hat}%
\end{equation}
and%
\begin{equation}
\Upsilon_{t}=\left(  \left(  \sigma_{x}^{1}(t)\right)  ^{\intercal}p_{t}%
+q_{t}^{1},\left(  \sigma_{x}^{2}(t)\right)  ^{\intercal}p_{t}+q_{t}%
^{2},\ldots,\left(  \sigma_{x}^{d}(t)\right)  ^{\intercal}p_{t}+q_{t}%
^{d}\right)  \text{ for }t\in\lbrack0,T].\label{def-capital-Gamma}%
\end{equation}

In view of the relation (\ref{decoupled-relation-y1-z1}), we introduce the
second-order variational equation as follows:
\begin{equation}
\left\{
\begin{array}
[c]{ll}%
dY_{2}(t)= & -\left\{  \left(  f_{x}(t)\right)  ^{\intercal}X_{2}%
(t)+f_{y}(t)Y_{2}(t)+\left(  f_{z}(t)\right)  ^{\intercal}Z_{2}(t)\right.  \\
& +\left[  \sum\limits_{i=1}^{d}\left(  \hat{\sigma}^{i}(t)\right)
^{\intercal}q_{t}^{i}+\hat{f}(t,\Delta1_{E_{\varepsilon}})\right]
1_{E_{\varepsilon}}(t)\\
& +\left.  \dfrac{1}{2}X_{1}^{\intercal}\left(  t\right)  \left(  I_{n}%
,p_{t},\Upsilon_{t}\right)  D^{2}f(t)\left(  I_{n},p_{t},\Upsilon_{t}\right)
^{\intercal}X_{1}\left(  t\right)  \right\}  dt+\ Z_{2}^{\intercal}%
(t)dW_{t},\\
Y_{2}(T)= & \left(  \Phi_{x}(\bar{X}_{T})\right)  ^{\intercal}X_{2}%
(T)+\dfrac{1}{2}\mathrm{tr}\left\{  \Phi_{xx}(\bar{X}_{T})X_{1}(T)X_{1}%
^{\intercal}(T)\right\}  ,
\end{array}
\right.  \label{new-form-y2}%
\end{equation}
where $\hat{f}(t,\Delta1_{E_{\varepsilon}})=f(t,\bar{X}_{t},\bar{Y}_{t},\bar{Z}%
_{t}+\Delta(t)1_{E_{\varepsilon}}(t),u_{t}^{\varepsilon})-f(t)$.
Similar to Lemma \ref{exist-y1-lem}, we have the following lemma.

\begin{lemma}
\label{exist-y2-lem} Suppose that Assumption \ref{assum-2} holds. Then there
exists a unique solution $(Y_{2}(\cdot),Z_{2}(\cdot))\in\bigcap_{\beta
>1}\left(  \mathcal{S}_{\mathcal{F}}^{\beta}([0,T];\mathbb{R})\times
\mathcal{M}_{\mathcal{F}}^{2,\beta}([0,T];\mathbb{R}^{d})\right)  $ to
(\ref{new-form-y2}).
\end{lemma}

In the following lemma, we estimate the orders of $Y_{2}(0)$ and
$Y^{\varepsilon}(0)-\bar{Y}(0)-Y_{1}(0)-Y_{2}(0)$.

\begin{lemma}
\label{est-second-order} Suppose that Assumption \ref{assum-2} holds. Then we
have%
\[
Y_{2}(0)=O(\varepsilon);\text{ \ }Y_{0}^{\varepsilon}-\bar{Y}_{0}%
-Y_{1}(0)-Y_{2}(0)=o(\varepsilon).
\]

\end{lemma}

\begin{proof}
For simplicity, we only prove the case $d=1$, and the proof for the case $d>1$
is similar. Set $\tilde{\Gamma}_{t}=\exp\left\{  \int_{0}^{t}f_{y}%
(s)ds\right\}  \Gamma_{t}$ for $t\in\lbrack0,T]$, where $\Gamma$ is defined in
Lemma \ref{est-one-order}. Note that $\tilde{\Gamma}$ satisfies the following SDE:
\begin{equation}\label{tilde-Gamma-SDE}
\tilde{\Gamma}_{t}=1+\int_{0}^{t}f_{y}(s)\tilde{\Gamma}_{s}ds+\sum_{i=1}%
^{d}\int_{0}^{t}f_{z_{i}}(s)\tilde{\Gamma}_{s}dW_{s}^{i},\text{ \ }t\in
\lbrack0,T].
\end{equation}
Thus, by applying It\^{o}'s formula to $\tilde{\Gamma}Y_{2}$ on
$[0,T]$ and taking expectation,
we get
\begin{equation}%
\begin{array}
[c]{rl}%
Y_{2}(0)= & \mathbb{E}\left[  \tilde{\Gamma}_{T}Y_{2}(T)+\int_{0}^{T}%
\tilde{\Gamma}_{t}\left\{  \left(  f_{x}(t)\right)  ^{\intercal}%
X_{2}(t)+\left[  \left(  \hat{\sigma}(t)\right)  ^{\intercal}q_{t}+\hat
{f}(t,\Delta1_{E_{\varepsilon}})\right]  1_{E_{\varepsilon}}(t)\right.
\right.  \\
& +\left.  \left.  \frac{1}{2}X_{1}^{\intercal}\left(  t\right)  \left[
I_{n},p_{t},\sigma_{x}^{\intercal}(t)p_{t}+q_{t}\right]  D^{2}f(t)\left[
I_{n},p_{t},\sigma_{x}^{\intercal}(t)p_{t}+q_{t}\right]  ^{\intercal}%
X_{1}\left(  t\right)  \right\}  dt\right]  .
\end{array}
\label{y2-explict}%
\end{equation}
We only estimate the most important and difficult terms in (\ref{y2-explict})
as follows.

Note that $\hat{f}(t,\Delta1_{E_{\varepsilon}})=\hat{f}(t,\Delta)1_{E_{\varepsilon}}(t)$. Due to Theorem \ref{state-eq-exist-th}, we have $\left\Vert \bar{Y}\right\Vert
_{\infty}<\infty$. Then, by $\left\Vert p \right\Vert_{\infty}<\infty$ and Assumption \ref{assum-2},
we have
\[%
\begin{array}
[c]{ll}%
\left\vert \hat{f}(t,\Delta)\right\vert  & \leq\left(  \left\vert f\left(
t,\bar{X}_{t},\bar{Y}_{t},\bar{Z}_{t}+\Delta(t),u_{t}\right)  -f\left(
t,\bar{X}_{t},\bar{Y}_{t},\bar{Z}_{t},u_{t}\right)  \right\vert +\left\vert
\hat{f}\left(  t\right)  \right\vert \right)  \\
& \leq C\left[  \left(  1+\left\vert \bar{Z}_{t}\right\vert +\left\vert
\Delta(t)\right\vert \right)  \left\vert \Delta(t)\right\vert +\left(
1+\left\vert \bar{Y}_{t}\right\vert +\left\vert \bar{Z}_{t}\right\vert
\right)  \right]  \\
& \leq C\left(  1+\left\vert \bar{X}_{t}\right\vert ^{2}+\left\vert
u_{t}\right\vert ^{2}+\left\vert \bar{u}_{t}\right\vert ^{2}+\left\vert
\bar{Z}_{t}\right\vert ^{2}\right)  .
\end{array}
\]
For any $\beta\in\left(  1,p_{f_{z}}^{{}}\right)  $, set $\beta^{\ast}%
=\beta(\beta-1)^{-1}$. It follows from H\"{o}lder's inequality, Doob's inequality and reverse
H\"{o}lder's inequality that%
\begin{equation}%
\begin{array}
[c]{cl}%
\mathbb{E}\left[  \int_{0}^{T}\Gamma_{t}\left\vert \bar{Z}_{t}\right\vert
^{2}dt\right]   & \leq\left(  \mathbb{E}\left[  \sup\limits_{t\in\lbrack
0,T]}\Gamma_{t}^{\beta}\right]  \right)  ^{\frac{1}{\beta}}\left(
\mathbb{E}\left[  \left(  \int_{0}^{T}\left\vert \bar{Z}_{t}\right\vert
^{2}dt\right)  ^{\beta^{\ast}}\right]  \right)  ^{\frac{1}{\beta^{\ast}}}\\
& \leq C\left(  \mathbb{E}\left[  \Gamma_{T}^{\beta}\right]  \right)
^{\frac{1}{\beta}}\left(  \mathbb{E}\left[  \left(  \int_{0}^{T}\left\vert
\bar{Z}_{t}\right\vert ^{2}dt\right)  ^{\beta^{\ast}}\right]  \right)
^{\frac{1}{\beta^{\ast}}}\\
& <\infty.
\end{array}
\label{gamma-Zbar2-finite}%
\end{equation}
Thus, we have%
\begin{equation}
\mathbb{E}\left[  \int_{E_{\varepsilon}}\Gamma_{t}\left\vert \bar{Z}%
_{t}\right\vert ^{2}dt\right]  \leq C\varepsilon.\label{gamma-Zbar}%
\end{equation}

Similar to (\ref{gamma-Zbar2-finite}), we have $\mathbb{E}\left[  \left(
\int_{0}^{T}\Gamma_{t}\left\vert q_{t}\right\vert ^{2}dt\right)  ^{\beta
}\right]  <\infty$. From the estimate (\ref{est-forward-expansion-2}), we get
\[%
\begin{array}
[c]{l}%
\mathbb{E}\left[  \int_{0}^{T}\tilde{\Gamma}_{t}\left\vert f_{zz}%
(t)\right\vert \left\vert q_{t}^{\intercal}X_{1}(t)\right\vert ^{2}dt\right]
\\
\leq C\mathbb{E}\left[  \sup\limits_{t\in\lbrack0,T]}\left\vert X_{1}%
(t)\right\vert ^{2}\left(  \int_{0}^{T}\Gamma_{t}\left\vert q_{t}\right\vert
^{2}dt\right)  \right]  \\
\leq C\left(  \mathbb{E}\left[  \left(  \int_{0}^{T}\Gamma_{t}\left\vert
q_{t}\right\vert ^{2}dt\right)  ^{\beta}\right]  \right)  ^{\frac{1}{\beta}%
}\left(  \mathbb{E}\left[  \sup\limits_{t\in\lbrack0,T]}\left\vert
X_{1}(t)\right\vert ^{2\beta^{\ast}}\right]  \right)  ^{\frac{1}{\beta^{\ast}%
}}\\
\leq C\varepsilon.
\end{array}
\]
Therefore, we finally obtain $Y_{2}(0)=O(\varepsilon)$.

Now we focus on the last estimate. We use the notation $\left(
\xi^{2,\varepsilon}(t),\eta^{2,\varepsilon}(t),\zeta^{2,\varepsilon
}(t)\right)  $ in the proof of Lemma \ref{est-one-order}. Let $\Theta(t,\Delta
1_{E_{\varepsilon}}):=(\bar{X}_{t},\bar{Y}_{t},\bar{Z}_{t}+\Delta
(t)1_{E_{\varepsilon}}(t))$, $$\left(
\xi^{3,\varepsilon}(t),\eta^{3,\varepsilon}(t),\zeta^{3,\varepsilon
}(t)\right)  :=\left(  \xi^{2,\varepsilon}(t)-X_{2}(t),\eta^{2,\varepsilon
}(t)-Y_{2}(t),\zeta^{2,\varepsilon}(t)-Z_{2}(t)\right)  $$ and define%
\[
\tilde{\Phi}_{xx}^{\varepsilon}(T)=2\int_{0}^{1}\int_{0}^{1}\theta\Phi
_{xx}(\bar{X}_{T}+\rho\theta(X_{T}^{\varepsilon}-\bar{X}_{T}))d\theta d\rho,
\]
\[
\tilde{D}^{2}f^{\varepsilon}(t)=2\int_{0}^{1}\int_{0}^{1}\theta D^{2}%
f(t,\Theta(t,\Delta1_{E_{\varepsilon}})+\rho\theta(\Theta_{t}^{\varepsilon
}-\Theta(t,\Delta1_{E_{\varepsilon}})),u_{t}^{\varepsilon})d\theta d\rho,
\]%
\[
\tilde{f}_{z}^{\varepsilon}(t,\Delta1_{E_{\varepsilon}})=\int_{0}^{1}%
f_{z}(t,\Theta(t,\Delta1_{E_{\varepsilon}})+\theta(\Theta_{t}^{\varepsilon
}-\Theta(t,\Delta1_{E_{\varepsilon}})),u_{t}^{\varepsilon})d\theta,
\]%
\[
\hat{f}_{z}(t,\Delta1_{E_{\varepsilon}})=f_{z}(t,\bar{X}_{t},\bar{Y}_{t}%
,\bar{Z}_{t}+\Delta(t)1_{E_{\varepsilon}}(t),u_{t}^{\varepsilon})-f_{z}(t).
\]
$\tilde{f}_{x}^{\varepsilon}(t,\Delta1_{E_{\varepsilon}})$, $\tilde{f}%
_{y}^{\varepsilon}(t,\Delta1_{E_{\varepsilon}})$,
$\hat{f}_{x}(t,\Delta1_{E_{\varepsilon}})$ and $\hat{f}_{y}(t,\Delta1_{E_{\varepsilon}})$ are defined similarly. Then,
by tedious calculation, we have%
\begin{equation}
\left\{
\begin{array}
[c]{lll}%
d\eta^{3,\varepsilon}(t) & = & -\left[  \left(  f_{x}(t)\right)  ^{\intercal
}\xi^{3,\varepsilon}(t)+f_{y}(t)\eta^{3,\varepsilon}(t)+f_{z}(t)\zeta
^{3,\varepsilon}(t)+A_{2}^{\varepsilon}(t)\right]  dt\\
&  & +\zeta^{3,\varepsilon}(t)dW_{t},\\
\eta^{3,\varepsilon}(T) & = & \left(  \Phi_{x}(\bar{X}_{T})\right)
^{\intercal}\xi^{3,\varepsilon}(T)+B_{2}^{\varepsilon}(T),
\end{array}
\right.  \label{eta-eps-3}%
\end{equation}
where%
\[%
\begin{array}
[c]{rl}%
A_{2}^{\varepsilon}(t)= & \left(  \hat{f}_{x}(t,\Delta1_{E_{\varepsilon}%
})\right)  ^{\intercal}\xi^{1,\varepsilon}(t)+\hat{f}_{y}(t,\Delta
1_{E_{\varepsilon}})\eta^{1,\varepsilon}(t)+\hat{f}_{z}(t,\Delta
1_{E_{\varepsilon}})(\zeta^{1,\varepsilon}(t)-\Delta(t)1_{E_{\varepsilon}%
}(t))\\
& -\dfrac{1}{2}\left(  X_{1}^{\intercal}\left(  t\right)  ,Y_{1}%
(t),Z_{1}(t)-\Delta(t)1_{E_{\varepsilon}}(t)\right)  D_{{}}^{2}f(t)\left(
X_{1}^{\intercal}\left(  t\right)  ,Y_{1}(t),Z_{1}(t)-\Delta
(t)1_{E_{\varepsilon}}(t)\right)  ^{\intercal}\\
& +\dfrac{1}{2}\left[  \left(  \xi^{1,\varepsilon}(t)\right)  ^{\intercal
},\eta^{1,\varepsilon}(t),\zeta^{1,\varepsilon}(t)-\Delta(t)1_{E_{\varepsilon
}}(t)\right]  \tilde{D}^{2}f^{\varepsilon}(t)\left[  \left(  \xi
^{1,\varepsilon}(t)\right)  ^{\intercal},\eta^{1,\varepsilon}(t),\zeta
^{1,\varepsilon}(t)-\Delta(t)1_{E_{\varepsilon}}(t)\right]  ^{\intercal},\\
B_{2}^{\varepsilon}(T)= & \dfrac{1}{2}\mathrm{tr}\left\{  \tilde{\Phi}%
_{xx}^{\varepsilon}(T)\xi^{1,\varepsilon}(T)\left(  \xi^{1,\varepsilon
}(T)\right)  ^{\intercal}\right\}  -\dfrac{1}{2}\mathrm{tr}\left\{  \Phi
_{xx}(\bar{X}_{T})X_{1}(T)\left(  X_{1}(T)\right)  ^{\intercal}\right\}  .
\end{array}
\]
Thus, similar to (\ref{y2-explict}), $\eta^{3,\varepsilon}(0)$ can be expressed as%
\begin{equation}
\eta^{3,\varepsilon}(0)=\mathbb{E}\left[  \tilde{\Gamma}_{T}\eta
^{3,\varepsilon}(T)+\int_{0}^{T}\tilde{\Gamma}_{t}\left[  \left(
f_{x}(t)\right)  ^{\intercal}\xi^{3,\varepsilon}(t)+A_{2}^{\varepsilon
}(t)\right]  dt\right]  .\label{eta3-explicit}%
\end{equation}
We only estimate the most important and difficult terms in
(\ref{eta3-explicit}) as follows.

Recall that the relationship between $\left( Y_{1}(\cdot), Z_{1}(\cdot) \right)$ and $X_{1}(\cdot)$ is obtained in (\ref{decoupled-relation-y1-z1}).
This implies
\[
\zeta^{1,\varepsilon}(t)-\Delta(t)1_{E_{\varepsilon}}(t)=\zeta^{2,\varepsilon
}(t)+\left(  p_{t}^{\intercal}\sigma_{x}(t)+q_{t}^{\intercal}\right)  X_{1}(t).
\]
Note that $\hat{f}_{z}(t,\Delta1_{E_{\varepsilon}})=\hat{f}_{z}(t,\Delta)1_{E_{\varepsilon}}(t)$.
Then, by the definition (\ref{f-fw-hat}) and Assumption (\ref{assum-2}), we have
\[
\left\vert \hat{f}_{z}(t,\Delta1_{E_{\varepsilon}})\right\vert =\left\vert
\hat{f}_{z}(t,\Delta)\right\vert 1_{E_{\varepsilon}}(t)\leq C\left(
1+\left\vert \bar{Z}_{t}\right\vert +\left\vert \bar{X}_{t}\right\vert
+\left\vert \bar{u}_{t}\right\vert +\left\vert u_{t}\right\vert \right)
1_{E_{\varepsilon}}(t)
\]
and%
\[%
\begin{array}
[c]{l}%
\left\vert \tilde{f}_{zz}^{\varepsilon}(t)|\zeta^{1,\varepsilon}%
(t)-\Delta(t)1_{E_{\varepsilon}}(t)|^{2}-f_{zz}(t)\left\vert Z_{1}%
(t)-\Delta(t)1_{E_{\varepsilon}}(t)\right\vert ^{2}\right\vert \\
\leq\left\vert \tilde{f}_{zz}^{\varepsilon}(t)\zeta^{2,\varepsilon}(t)\left[
\zeta^{1,\varepsilon}(t)-\Delta(t)1_{E_{\varepsilon}}(t)+\left(
p_{t}^{\intercal}\sigma_{x}(t)+q_{t}^{\intercal}\right)  X_{1}(t)\right]
\right\vert \\
\ \ +\left\vert \tilde{f}_{zz}^{\varepsilon}(t)-f_{zz}(t)\right\vert
\left\vert \left(  p_{t}^{\intercal}\sigma_{x}(t)+q_{t}^{\intercal}\right)
X_{1}(t)\right\vert ^{2}.
\end{array}
\]

For any $\beta\in\left(  1\vee2p_{f_{z}}^{-1},2\right)  $, set $\beta^{\ast
}=\beta(\beta-1)^{-1}$. It follows from H\"{o}lder's inequality, Doob's inequality, reverse
H\"{o}lder's inequality and the energy inequality for $\bar{Z} \cdot W$ that
\begin{equation}%
\begin{array}
[c]{l}%
\mathbb{E}\left[  \left(  \int_{0}^{T}\left(  \Gamma_{t}\right)  ^{\frac
{2}{\beta^{\ast}}}\left\vert \bar{Z}_{t}\right\vert ^{2}dt\right)
^{\frac{\beta^{\ast}}{2}}\right]  \\
\leq\mathbb{E}\left[  \sup\limits_{t\in\lbrack0,T]}\Gamma_{t}\left(  \int%
_{0}^{T}\left\vert \bar{Z}_{t}\right\vert ^{2}dt\right)  ^{\frac{\beta^{\ast}%
}{2}}\right]  \\
\leq\left(  \mathbb{E}\left[  \sup\limits_{t\in\lbrack0,T]}\left(  \Gamma
_{t}\right)  ^{\frac{1+p_{f_{z}}^{{}}}{2}}\right]  \right)  ^{\frac
{2}{1+p_{f_{z}}^{{}}}}\left(  \mathbb{E}\left[  \left(  \int_{0}^{T}\left\vert
\bar{Z}_{t}\right\vert ^{2}dt\right)  ^{\frac{\beta^{\ast}(p_{f_{z}}^{{}}%
+1)}{2(p_{f_{z}}^{{}}-1)}}\right]  \right)  ^{\frac{p_{f_{z}}^{{}}-1}%
{p_{f_{z}}^{{}}+1}}\\
<\infty.
\end{array}
\label{gamma-Zbar-square-finite}%
\end{equation}
Hence, by the estimate (\ref{est-eta2-zeta2}), we get%
\begin{equation}%
\begin{array}
[c]{l}%
\mathbb{E}\left[  \int_{E_{\varepsilon}}\tilde{\Gamma}_{t}\left\vert \bar
{Z}_{t}\right\vert \left\vert \zeta^{2,\varepsilon}(t)\right\vert dt\right]
\\
\leq C\mathbb{E}\left[  \left(  \int_{E_{\varepsilon}}\left(  \Gamma
_{t}\right)  ^{\frac{2}{\beta^{\ast}}}\left\vert \bar{Z}_{t}\right\vert
^{2}dt\right)  ^{\frac{1}{2}}\left(  \int_{0}^{T}\left(  \Gamma_{t}\right)
^{\frac{2}{\beta}}\left\vert \zeta^{2,\varepsilon}(t)\right\vert
^{2}dt\right)  ^{\frac{1}{2}}\right]  \\
\leq C\left(  \mathbb{E}\left[  \left(  \int_{E_{\varepsilon}}\left(
\Gamma_{t}\right)  ^{\frac{2}{\beta^{\ast}}}\left\vert \bar{Z}_{t}\right\vert
^{2}dt\right)  ^{\frac{\beta^{\ast}}{2}}\right]  \right)  ^{\frac{1}%
{\beta^{\ast}}}\left(  \mathbb{E}\left[  \left(  \int_{0}^{T}\left(
\Gamma_{t}\right)  ^{\frac{2}{\beta}}\left\vert \zeta^{2,\varepsilon
}(t)\right\vert ^{2}dt\right)  ^{\frac{\beta}{2}}\right]  \right)  ^{\frac
{1}{\beta}}\\
\leq C\left(  \mathbb{E}\left[  \left(  \int_{E_{\varepsilon}}\left(
\Gamma_{t}\right)  ^{\frac{2}{\beta^{\ast}}}\left\vert \bar{Z}_{t}\right\vert
^{2}dt\right)  ^{\frac{\beta^{\ast}}{2}}\right]  \right)  ^{\frac{1}%
{\beta^{\ast}}}\varepsilon\\
=o(\varepsilon).
\end{array}
\label{est-gamma-Zbar-zeta2}%
\end{equation}

Set $p^{\prime}=\frac{1+p_{f_{z}}^{{}}}{2}$, $q^{\prime}=p^{\prime}(p^{\prime
}-1)^{-1}$. Since $q(\cdot) \in \bigcap_{\beta
>1}\mathcal{M}_{\mathcal{F}}^{2,\beta}([0,T];\mathbb{R}_{{}}^{n\times d})$,
by reverse H\"{o}lder's inequality,
we can show that $\mathbb{E}\left[  \left(  \int_{0}%
^{T}\left(  \Gamma_{t}\right)  ^{\frac{2}{\beta^{\ast}}}\left\vert
q_{t}\right\vert ^{2}dt\right)  ^{\frac{\beta^{\ast}p^{\prime}}{2}}\right]
<\infty$ like (\ref{gamma-Zbar-square-finite}). Then, we have
\[%
\begin{array}
[c]{l}%
\mathbb{E}\left[  \int_{0}^{T}\tilde{\Gamma}_{t}\left\vert \zeta
^{2,\varepsilon}(t)\right\vert \left\vert q_{t}\right\vert \left\vert
X_{1}(t)\right\vert dt\right]  \\
\leq C\mathbb{E}\left[  \sup\limits_{t\in\lbrack0,T]}\left\vert X_{1}%
(t)\right\vert \left(  \int_{0}^{T}\left(  \Gamma_{t}\right)  ^{\frac{2}%
{\beta^{\ast}}}\left\vert q_{t}\right\vert ^{2}dt\right)  ^{\frac{1}{2}%
}\left(  \int_{0}^{T}\left(  \Gamma_{t}\right)  ^{\frac{2}{\beta}}\left\vert
\zeta^{2,\varepsilon}(t)\right\vert ^{2}dt\right)  ^{\frac{1}{2}}\right]  \\
\leq C\left(  \mathbb{E}\left[  \sup\limits_{t\in\lbrack0,T]}\left\vert
X_{1}(t)\right\vert ^{\beta^{\ast}q^{\prime}}\right]  \right)  ^{\frac
{1}{\beta^{\ast}q^{\prime}}}\left(  \mathbb{E}\left[  \left(  \int_{0}%
^{T}\left(  \Gamma_{t}\right)  ^{\frac{2}{\beta^{\ast}}}\left\vert
q_{t}\right\vert ^{2}dt\right)  ^{\frac{\beta^{\ast}p^{\prime}}{2}}\right]
\right)  ^{\frac{1}{\beta^{\ast}p^{\prime}}}\\
\ \ \times\left(  \mathbb{E}\left[  \left(  \int_{0}^{T}\left(  \Gamma
_{t}\right)  ^{\frac{2}{\beta}}\left\vert \zeta^{2,\varepsilon}(t)\right\vert
^{2}dt\right)  ^{\frac{\beta}{2}}\right]  \right)  ^{\frac{1}{\beta}}\\
\leq C\varepsilon^{\frac{3}{2}}.
\end{array}
\]
Finally, we obtain $Y^{\varepsilon}(0)-\bar{Y}(0)-Y_{1}(0)-Y_{2}(0)=\eta^{3,\varepsilon}(0)=o\left(  \varepsilon \right)$.
The proof is complete.
\end{proof}

\subsection{Maximum principle}

By Lemma \ref{relation-y1z1} and $X_{1}(0)=0$, we have $Y_{1}(0)=0$. Thus, by
Lemma \ref{est-second-order}, we obtain%
\begin{equation}
J(u^{\varepsilon}(\cdot))-J(\bar{u}(\cdot))=Y^{\varepsilon}(0)-\bar
{Y}(0)=Y_{2}(0)+o(\varepsilon).\label{costfunc-Y2-oeps}%
\end{equation}
In order to obtain $Y_{2}(0)$, we introduce the second-order adjoint
equation:
\begin{equation}%
\begin{array}
[c]{rl}%
P_{t}= & \Phi_{xx}(\bar{X}_{T})+%
{\displaystyle\int_{t}^{T}}
\left\{  f_{y}(s)P_{s}+\sum\limits_{i=1}^{d}f_{z_{i}}(s)\left[  \left(
\sigma_{x}^{i}(s)\right)  ^{\intercal}P_{s}+P_{s}^{\intercal}\sigma_{x}%
^{i}(s)\right]  \right.  \\
& +b_{x}^{\intercal}(s)P_{s}+P_{s}^{\intercal}b_{x}(s)+\sum\limits_{i=1}%
^{d}\left(  \sigma_{x}^{i}(s)\right)  ^{\intercal}P_{s}\sigma_{x}^{i}%
(s)+\sum\limits_{i=1}^{d}f_{z_{i}}(s)Q_{s}^{i}\\
& \left.  +\sum\limits_{i=1}^{d}\left[  \left(  \sigma_{x}^{i}(s)\right)
^{\intercal}Q_{s}^{i}+\left(  Q_{s}^{i}\right)  ^{\intercal}\sigma_{x}%
^{i}(s)\right]  +\varphi_{s}\right\}  ds-\sum\limits_{i=1}^{d}%
{\displaystyle\int_{t}^{T}}
Q_{s}^{i}dW_{s}^{i},
\end{array}
\label{eq-P}%
\end{equation}
where
\begin{equation}
\varphi_{t}=\sum\limits_{i=1}^{n}b_{xx}^{i}(t)p_{t}^{i}+\sum\limits_{i=1}%
^{n}\sum\limits_{j=1}^{d}\sigma_{xx}^{ij}(t)\left(  f_{z_{j}}(t)p_{t}%
^{i}+q_{t}^{ij}\right)  +\left(  I_{n},p_{t},\Upsilon_{t}\right)
D^{2}f(t)\left(  I_{n},p_{t},\Upsilon_{t}\right)  ^{\intercal}.
\label{2nd-adj-eq-multi-data}%
\end{equation}
Recall that $\Upsilon$ is defined in (\ref{def-capital-Gamma}), $p(\cdot)$ is bounded and $q(\cdot)\in\bigcap_{\beta>1}\mathcal{M}_{\mathcal{F}}^{2,\beta}
([0,T];\mathbb{R}_{{}}^{n\times d})$ from Lemma \ref{adj-1st-lem}. Then, by Assumption \ref{assum-2} and the energy inequality for $\bar{Z} \cdot W$, we have $\mathbb{E}\left[  \left(  \int_{0}^{T}\left\vert
\varphi_{t}\right\vert dt\right)  ^{\beta}\right]  <\infty$ for all $\beta>1$.
The matrix-valued BSDE (\ref{eq-P}) can be rewritten in the vector-valued
form, which belongs to the class of BSDEs (\ref{multi-linear-BSDE}).
Therefore, we have the following lemma immediately.

\begin{lemma}
\label{adj-2nd-lem} Suppose that Assumption \ref{assum-2} holds. Then the BSDE
(\ref{eq-P}) admits a unique solution $\left(  P(\cdot),Q(\cdot)\right)  \in\bigcap_{\beta>1}\left\{  \mathcal{S}%
_{\mathcal{F}}^{\beta}([0,T];\mathbb{S}_{{}}^{n\times n})\times\left(
\mathcal{M}_{\mathcal{F}}^{2,\beta}([0,T];\mathbb{S}_{{}}^{n\times n})\right)
^{d}\right\}  $.
\end{lemma}

Now we deduce the relationship between $(Y_{2}(\cdot),Z_{2}(\cdot))$ and
$X_{2}(\cdot)$. We introduce the following auxiliary equation:
\begin{equation}%
\begin{array}
[c]{cl}%
\hat{Y}_{t}= &
{\displaystyle\int_{t}^{T}}
\left\{  f_{y}(s)\hat{Y}_{s}+\left(  f_{z}(s)\right)  ^{\intercal}\hat{Z}%
_{s}+\left[  p_{s}^{\intercal}\hat{b}(s)+\sum\limits_{i=1}^{d}\left(
q_{s}^{i}\right)  ^{\intercal}\hat{\sigma}^{i}(s)\right]  1_{E_{\varepsilon}%
}(s)\right.  \\
& +\left.  \left[  \hat{f}(s,\Delta)+\dfrac{1}{2}\sum\limits_{i=1}^{d}\left(
\hat{\sigma}^{i}(s)\right)  ^{\intercal}P_{s}\hat{\sigma}^{i}(s)\right]
1_{E_{\varepsilon}}(s)\right\}  ds-%
{\displaystyle\int_{t}^{T}}
\left(  \hat{Z}_{s}\right)  ^{\intercal}dW_{s}.
\end{array}
\label{eq-y-hat}%
\end{equation}

\begin{lemma}
\label{relation-y2z2} Suppose that Assumption \ref{assum-2} holds. Then the BSDE
(\ref{eq-y-hat}) admits a unique solution $\left(  \hat{Y},\hat{Z}\right)
\in\bigcap_{\beta>1}\left(  \mathcal{S}_{\mathcal{F}}^{\beta}([0,T];\mathbb{R}%
)\times\mathcal{M}_{\mathcal{F}}^{2,\beta}([0,T];\mathbb{R}^{d})\right)  $.
Moreover, we have%
\begin{equation}%
\begin{array}
[c]{rl}%
Y_{2}(t) & =\hat{Y}_{t}+p_{t}^{\intercal}X_{2}(t)+\dfrac{1}{2}X_{1}%
^{\intercal}(t)P_{t}X_{1}(t),\\
Z_{2}^{i}(t) & =\hat{Z}_{t}^{i}+\mathbf{\tilde{Z}}^{i}\mathbf{(t)}\text{ \ for
\ }i=1,2,\ldots,d,
\end{array}
\label{relation-y2-z2}%
\end{equation}
where%
\[%
\begin{array}
[c]{cl}%
\mathbf{\tilde{Z}}^{i}\mathbf{(t)}= & \left[  p_{t}^{\intercal}\sigma_{x}%
^{i}(t)+\left(  q_{t}^{i}\right)  ^{\intercal}\right]  X_{2}(t)\\
& +\dfrac{1}{2}X_{1}^{\intercal}(t)\left[  \left(  \sigma_{x}^{i}(t)\right)
^{\intercal}P_{t}+P_{t}^{\intercal}\sigma_{x}^{i}(t)+Q_{t}^{i}+\sum
\limits_{j=1}^{d}\sigma_{xx}^{ij}(t)p_{t}^{j}\right]  X_{1}(t)\\
& +\left\{  \dfrac{1}{2}\left[  \left(  \hat{\sigma}^{i}(t)\right)
^{\intercal}P_{t}+P_{t}^{\intercal}\hat{\sigma}^{i}(t)\right]  +p_{t}%
^{\intercal}\hat{\sigma}_{x}^{i}(t)\right\}  X_{1}(t)1_{E_{\varepsilon}}(t).
\end{array}
\]

\end{lemma}

\begin{proof}
By Proposition \ref{est-linear-prop}, (\ref{eq-y-hat}) has a unique solution $\left(
\hat{Y},\hat{Z}\right)  \in \bigcap_{\beta>1}\left(  \mathcal{S}_{\mathcal{F}%
}^{\beta}([0,T];\mathbb{R})\times\mathcal{M}_{\mathcal{F}}^{2,\beta
}([0,T];\mathbb{R}^{d})\right)  $ since%
\[
\mathbb{E}\left[  \left(  \int_{0}^{T}\left\vert p_{t}^{\intercal}\hat
{b}(t)+\sum\limits_{i=1}^{d}\left(  q_{t}^{i}\right)  ^{\intercal}\hat{\sigma
}^{i}(t)+\hat{f}(t,\Delta)+\dfrac{1}{2}\sum\limits_{i=1}^{d}\left(
\hat{\sigma}^{i}(t)\right)  ^{\intercal}P_{t}\hat{\sigma}^{i}(t)\right\vert
dt\right)  ^{\beta}\right]  <\infty
\]
for all $\beta>1$. Applying It\^{o}'s formula to $p_{t}^{\intercal}%
X_{2}(t)+\frac{1}{2}X_{1}^{\intercal}(t)P_{t}X_{1}(t)+\hat{Y}_{t}$, we obtain
the relationship (\ref{relation-y2-z2}).
\end{proof}

Applying It\^{o}'s formula to $\tilde{\Gamma}\hat{Y}$ and taking expectation, we obtain
\begin{equation}%
\begin{array}
[c]{rl}%
\hat{Y}_{0}= & \mathbb{E}\left[
{\displaystyle\int_{0}^{T}}
\tilde{\Gamma}_{t}\left\{  p_{t}^{\intercal}\hat{b}(t)+\sum\limits_{i=1}%
^{d}\left(  q_{t}^{i}\right)  ^{\intercal}\hat{\sigma}^{i}(t)+\hat{f}%
(t,\Delta)\right.  \right.  \\
& +\left.  \left.  \dfrac{1}{2}\sum\limits_{i=1}^{d}\left(  \hat{\sigma}%
^{i}(t)\right)  ^{\intercal}P_{t}\hat{\sigma}^{i}(t)\right\}
1_{E_{\varepsilon}}(t)dt\right]  ,
\end{array}
\label{Y-hat-0}%
\end{equation}
where $\tilde{\Gamma}$ satisfies the SDE (\ref{tilde-Gamma-SDE}). Combing the relationship
(\ref{costfunc-Y2-oeps}) with (\ref{relation-y2-z2}), we deduce that
\begin{equation}
J(u^{\varepsilon}(\cdot))-J(\bar{u}(\cdot))=\hat{Y}_{0}+o(\varepsilon
)\geq0.\label{costfunc-Yhat-oeps}%
\end{equation}
Define the Hamiltonian $\mathcal{H}:[0,T]\times\Omega\times\mathbb{R}%
^{n}\times\mathbb{R}\times\mathbb{R}^{d}\times U\times\mathbb{R}^{n}\times$
$\mathbb{R}^{n\times d}\times\mathbb{R}^{n\times n}\longmapsto\mathbb{R}$ by%
\begin{equation}%
\begin{array}
[c]{rl}%
\mathcal{H}(t,x,y,z,u,p,q,P)= & p^{\intercal}b(t,x,u)+\sum\limits_{i=1}%
^{d}\left(  q^{i}\right)  ^{\intercal}\sigma^{i}(t,x,u)+f(t,x,y,z+\Delta
(t),u)\\
& +\dfrac{1}{2}\sum\limits_{i=1}^{d}\left(  \sigma^{i}(t,x,u)-\sigma
^{i}(t,\bar{X}_{t},\bar{u}_{t})\right)  ^{\intercal}P\left(  \sigma
^{i}(t,x,u)-\sigma^{i}(t,\bar{X}_{t},\bar{u}_{t})\right)  ,
\end{array}
\label{def-H}%
\end{equation}
where
\[
\Delta(t):=\left(  \left(  \sigma^{1}(t,x,u)-\sigma^{1}(t,\bar{X}_{t},\bar
{u}_{t})\right)  ^{\intercal}p,\ldots,\left(  \sigma^{d}(t,x,u)-\sigma
^{d}(t,\bar{X}_{t},\bar{u}_{t})\right)  ^{\intercal}p\right)  ^{\intercal}.
\]
By (\ref{def-H}), we get
\begin{equation}%
\begin{array}
[c]{l}%
\mathcal{H}(t,\bar{X}_{t},\bar{Y}_{t},\bar{Z}_{t},u_{t},p_{t},q_{t}%
,P_{t})-\mathcal{H}(t,\bar{X}_{t},\bar{Y}_{t},\bar{Z}_{t},\bar{u}_{t}%
,p_{t},q_{t},P_{t})\\
=p_{t}^{\intercal}\hat{b}(t)+\sum\limits_{i=1}^{d}\left(  q_{t}^{i}\right)
^{\intercal}\hat{\sigma}^{i}(t)+\hat{f}(t,\Delta)+\dfrac{1}{2}\sum
\limits_{i=1}^{d}\left(  \hat{\sigma}^{i}(t)\right)  ^{\intercal}P_{t}%
\hat{\sigma}^{i}(t).
\end{array}
\label{Hamiltonian-subtract}%
\end{equation}
It follows from (\ref{Y-hat-0}), (\ref{costfunc-Yhat-oeps}) and
(\ref{Hamiltonian-subtract}) that
\[
J(u^{\varepsilon}(\cdot))-J(\bar{u}(\cdot))=\mathbb{E}\left[  \int_{0}%
^{T}\tilde{\Gamma}_{t}\left\{  \mathcal{H}(t,\bar{X}_{t},\bar{Y}_{t},\bar
{Z}_{t},u_{t},p_{t},q_{t},P_{t})-\mathcal{H}(t,\bar{X}_{t},\bar{Y}_{t},\bar
{Z}_{t},\bar{u}_{t},p_{t},q_{t},P_{t})\right\}  dt\right]  +o(\varepsilon).
\]
Since $\tilde{\Gamma}_{t}>0$ for $t\in\lbrack0,T]$, we obtain the following
maximum principle immediately.

\begin{theorem}
\label{Th-MP} Suppose that Assumption \ref{assum-2} holds. Let $\bar{u}%
(\cdot)\in\mathcal{U}[0,T]$ be optimal and $(\bar{X}(\cdot),\bar{Y}%
(\cdot),\bar{Z}(\cdot))$ be the corresponding state trajectories of
(\ref{state-eq}). Then the following stochastic maximum principle holds:
\begin{equation}
\mathcal{H}(t,\bar{X}_{t},\bar{Y}_{t},\bar{Z}_{t},u,p_{t},q_{t},P_{t}%
)\geq\mathcal{H}(t,\bar{X}_{t},\bar{Y}_{t},\bar{Z}_{t},\bar{u}_{t},p_{t}%
,q_{t},P_{t}),\ \ \ \forall u\in U,\ \text{ }dt \otimes d\mathbb{P}-a.e.,
\label{mp-1}%
\end{equation}
where $\mathcal{H}$ is defined in (\ref{def-H}) and $(p\left(  \cdot\right)  ,q\left(  \cdot\right)  )$, $\left(  P\left(
\cdot\right)  ,Q\left(  \cdot\right)  \right)  $ satisfy (\ref{eq-p}),
(\ref{eq-P}) respectively.
\end{theorem}

Now we consider a special case that the control domain $U$ is convex, and further assume that:

\begin{assumption}\label{assum-5}
(i) The control domain $U$ is a convex body in $\mathbb{R}^{k}$.

(ii) $b$, $\sigma$ and $f$ are continuous differentiable with respect to $u$; $b_{u}$ and $\sigma_{u}$ are bounded; there exists a constant $L_{4} > 0$ such that
$$\left \vert f_{u}(t,x,y,z,u) \right \vert \leq L_{4} \left ( 1 + \left \vert y \right \vert + \left \vert z \right \vert \right ).$$
\end{assumption}

For $i=1,2,\ldots,d$, set%
\[%
\begin{array}
[c]{cc}%
b_{u}(\cdot)=\left(
\begin{array}
[c]{ccc}%
b_{u_{1}}^{1}(\cdot) & \cdots & b_{u_{k}}^{1}(\cdot)\\
\vdots & \ddots & \vdots\\
b_{u_{1}}^{n}(\cdot) & \cdots & b_{u_{k}}^{n}(\cdot)
\end{array}
\right)  , & \sigma_{u}^{i}(\cdot)=\left(
\begin{array}
[c]{ccc}%
\sigma_{u_{1}}^{1i}(\cdot) & \cdots & \sigma_{u_{k}}^{1i}(\cdot)\\
\vdots & \ddots & \vdots\\
\sigma_{u_{1}}^{ni}(\cdot) & \cdots & \sigma_{u_{k}}^{ni}(\cdot)
\end{array}
\right)  .
\end{array}
\]
We have the following local stochastic maximum principle.

\begin{corollary}\label{coro-local-SMP}
Suppose Assumptions \ref{assum-2} and \ref{assum-5} hold. Let $\bar{u}(\cdot)$ be an optimal control,
$(\bar{X}(\cdot),\bar{Y}(\cdot),\bar{Z}(\cdot))$ be the corresponding state trajectories.
Then the global SMP (\ref{mp-1}) degenerates into the following local version:
\begin{equation}%
\begin{array}
[c]{l}%
\left[  \sum\limits_{i=1}^{d}f_{z_{i}}(t,\bar{X}_{t},\bar{Y}_{t},\bar{Z}%
_{t},\bar{u}_{t})p_{t}^{\intercal}\sigma_{u}^{i}(t,\bar{X}_{t},\bar{u}%
_{t})+f_{u}^{\intercal}(t,\bar{X}_{t},\bar{Y}_{t},\bar{Z}_{t},\bar{u}%
_{t})\right.  \\
+\left.  p_{t}^{\intercal}b_{u}(t,\bar{X}_{t},\bar{u}_{t})+\sum\limits_{i=1}%
^{d}\left(  q_{t}^{i}\right)  ^{\intercal}\sigma_{u}^{i}(t,\bar{X}_{t},\bar
{u}_{t})\right]  (u-\bar{u}_{t})\geq0,\text{ \ }\\
\text{
\ \ \ \ \ \ \ \ \ \ \ \ \ \ \ \ \ \ \ \ \ \ \ \ \ \ \ \ \ \ \ \ \ \ \ \ \ }%
\forall u\in U,\text{ }dt\otimes d\mathbb{P}-a.e.,
\end{array}
\label{local-SMP}%
\end{equation}
where $(p\left(  \cdot\right)  ,q\left(  \cdot\right)  )$, $\left(  P\left(
\cdot\right)  ,Q\left(  \cdot\right)  \right)  $ satisfy (\ref{eq-p}),
(\ref{eq-P}) respectively.
\end{corollary}

\subsection{An example}
We provide an explicitly solvable problem to illustrate Theorem
\ref{Th-MP}. Let $n=d=k=1$ and the control domain $U=\left\{  0,1\right\}  $.

Consider the following control system:
\[
\left\{
\begin{array}
[c]{rl}%
X_{t}^{u}= & \int_{0}^{t}u_{s}dW_{s},\\
Y_{t}^{u}= & \Phi(X_{T}^{u})+\int_{t}^{T}\left(  g(Z_{s}^{u})+\left\vert
u_{s}\right\vert ^{2}\right)  ds-\int_{t}^{T}Z_{s}^{u}dW_{s},
\end{array}
\right.
\]
where $\Phi$ and $g$ are two deterministic functions defined on $\mathbb{R}%
$\ satisfying

(i) $\Phi(0)=0$; $0\leq\Phi^{\prime}(x)\leq1$ and $\left\vert \Phi
^{\prime\prime}(x)\right\vert <1$ for any $x \in \mathbb{R}$.

(ii) $g(0)=0$, $g(1)>0$ and $g^{\prime}(0)<0$; $\left\vert g(z)\right\vert
\leq\frac{1}{2}$ when $z\in\lbrack0,1]$.

We further suppose that Assumption \ref{assum-2} holds.
For instance, $\Phi(x)=\arctan x$ and $g(z)=\frac{z}%
{2}(2\left\vert z\right\vert -1)$ satisfy the above conditions. Our goal is
to minimize the cost functional $J(u(\cdot)):=Y_{0}^{u}$ over $\mathcal{U}%
[0,T]$. Recall that (\ref{control-integrable}) holds naturally in this case
since $U$ is bounded. We can show that $\bar{u}(\cdot)=0$ is the unique
optimal control, and $(\bar{X}(\cdot),\bar{Y}(\cdot),\bar{Z}(\cdot))=\left(
0,0,0\right)  $ are the corresponding unique optimal state trajectories.

Actually, for any $u(\cdot)\in\mathcal{U}[0,T]$, it follows from It\^{o}'s
formula that%
\begin{equation}%
\begin{array}
[c]{rl}
& Y_{t}^{u}-\Phi(X_{t}^{u})\\
= & \int_{t}^{T}\left\{  g(Z_{s}^{u})+\left[  1+\frac{1}{2}\Phi^{\prime\prime
}(X_{s}^{u})\right]  \left\vert u_{s}\right\vert ^{2}\right\}  ds-\int_{t}%
^{T}(Z_{s}^{u}-\Phi^{\prime}(X_{s}^{u})u_{s})dW_{s}\\
= & \int_{t}^{T}\left\{  \alpha_{s}^{u}(Z_{s}^{u}-\Phi^{\prime}(X_{s}%
^{u})u_{s})+g(\Phi^{\prime}(X_{s}^{u})u_{s})+\left[  1+\frac{1}{2}\Phi
^{\prime\prime}(X_{s}^{u})\right]  \left\vert u_{s}\right\vert ^{2}\right\}
ds\\
& -\int_{t}^{T}(Z_{s}^{u}-\Phi^{\prime}(X_{s}^{u})u_{s})dW_{s},
\end{array}
\label{ex-Y-Phi-dP}%
\end{equation}
where $\alpha_{s}^{u}=\int_{0}^{1}g^{\prime}(\Phi^{\prime}(X_{s}^{u})u_{s}%
+\theta(Z_{s}^{u}-\Phi^{\prime}(X_{s}^{u})u_{s}))d\theta$. For each
$t\in\lbrack0,T]$, by Assumption \ref{assum-2} (iii), the boundedness of $\Phi^{\prime}$
and$\ U$, we have $\left\vert \alpha_{t}^{u}\right\vert \leq C\left(
1+\left\vert Z_{t}^{u}\right\vert \right)  $. Then, by Theorem \ref{est-quad-qfbsde}, we have
$\alpha^{u}\cdot W\in BMO$ since $Z_{{}}^{u}\cdot W\in BMO$. It follows
immediately that $\mathcal{E}\left(  \alpha^{u}\cdot W\right)  $ is a
uniformly integrable martingale. By Girsanov's theorem, we can define a new
probability $\mathbb{P}^{u}$ and a Brownian motion $W^{u}$ under
$\mathbb{P}^{u}$ by%
\[
d\mathbb{P}^{u}=\mathcal{E}\left(  \alpha^{u}\cdot W\right)  d\mathbb{P}%
;\text{ \ }W_{t}^{u}=W_{t}-\int_{0}^{t}\alpha_{s}^{u}ds, \text{ } t \in [0,T].
\]
Denote by $\mathbb{E}^{u}\left[  \cdot\right]  $ the mathematical expectation
corresponding to $\mathbb{P}^{u}$. Thus, (\ref{ex-Y-Phi-dP}) can be rewritten as%
\begin{equation}
Y_{t}^{u}-\Phi(X_{t}^{u})=\int_{t}^{T}\left\{  g(\Phi^{\prime}(X_{s}^{u}%
)u_{s})+\left[  1+\frac{1}{2}\Phi^{\prime\prime}(X_{s}^{u})\right]  \left\vert
u_{s}\right\vert ^{2}\right\}  ds-\int_{t}^{T}(Z_{s}^{u}-\Phi^{\prime}%
(X_{s}^{u})u_{s})dW_{s}^{u}.\label{ex-Y-Phi-dPu}%
\end{equation}
Since $Z_{{}}^{u}\cdot W\in BMO$, $\Phi^{\prime}$ and$\ U$ is bounded, by
reverse H\"{o}lder's inequality for $\mathcal{E}\left(  \alpha^{u}\cdot
W\right)  $ and H\"{o}lder's inequality, we get $\mathbb{E}^{u}\left[
\sqrt{\int_{0}^{T}\left\vert Z_{t}^{u}-\Phi^{\prime}(X_{t}^{u})u_{t}%
\right\vert ^{2}dt}\right]  <\infty$. Thus the stochastic integral in (\ref{ex-Y-Phi-dPu}) is a
true martingale under $\mathbb{P}^{u}$.

Due to the condition (i) and (ii), we get $g(\Phi^{\prime}(x)u)+\left[  1+\frac
{1}{2}\Phi^{\prime\prime}(x)\right]  \left\vert u\right\vert ^{2}\geq0$ for
any $(x,u)\in\mathbb{R}\times U$. In particular,%
\begin{equation}\label{strict-geq-0}
  g(\Phi^{\prime}(x)u)+\left[  1+\frac{1}{2}\Phi^{\prime\prime}(x)\right]
\left\vert u\right\vert ^{2}=0\text{ iff }u=0.
\end{equation}
Note that $\mathbb{P}^{u}$ is equivalent to $\mathbb{P}$. Therefore, if $\left(  \lambda\otimes\mathbb{P}\right)  \left(  \left\{
(t,\omega):u_{t}(\omega)\neq0\right\}  \right)  >0$, by taking
expectation $\mathbb{E}^{u}$ on both sides of (\ref{ex-Y-Phi-dPu}) when $t=0$, then it follows from $\Phi(0)=0$ and (\ref{strict-geq-0}) that
\[
Y_{0}^{u}=\mathbb{E}^{u}\left[
{\displaystyle\int_{0}^{T}}
\left\{  g(\Phi^{\prime}(X_{s}^{u})u_{s})+\left[  1+\frac{1}{2}\Phi
^{\prime\prime}(X_{s}^{u})\right]  \left\vert u_{s}\right\vert ^{2}\right\}
ds\right]  >0,
\]
where $\lambda$ represents the Lebesgue measure on $[0,T]$. This implies $\bar{u}(\cdot)=0$ is the unique optimal control.
In this case, we can calculate the adjoint processes $\left(  p(\cdot
),q(\cdot)\right)  =\left(  1,0\right)  $ and $\left(  P(\cdot),Q(\cdot
)\right)  =\left(  0,0\right)  $. By Theorem \ref{Th-MP}, we obtain the
following global SMP:
\[
g(u)  +\left\vert u\right\vert ^{2}\geq0,\text{ }\forall u\in
U,\text{ \ }dt\otimes d\mathbb{P}-a.e..
\]

\begin{remark}
It is worth pointing out that for a stochastic optimal control problem, if its control domain is replaced with its convex hull, then the global SMP will have a completely different expression.
For example, when we replace the control domain $U=\left\{  0,1\right\}  $ with its convex hull $U^{\prime}=[0,1]$ in the above example, Corollary \ref{coro-local-SMP} indicates that the optimal control $\bar{u} (\cdot)$ should satisfy
\begin{equation}\label{eg-local-SMP}
\left[ g^{\prime}(\bar{Z}_{t})p_{t} + q_{t} + 2\bar{u}_{t}\right] (1-\bar{u}_{t}) \geq 0.
\end{equation}
For this case, if $\bar{u}(\cdot)=0$, then $\left( \bar{Z}(\cdot), p(\cdot), q(\cdot) \right) = \left(0,1,0 \right)$ and (\ref{eg-local-SMP}) becomes $g^{\prime}(0)\geq 0$
which contradicts our assumption $g^{\prime}(0)<0$. Thus, $\bar{u}(\cdot)=0$ can not be an optimal control when the control domain is $U^{\prime}=[0,1]$.
\end{remark}

\section{A sufficient condition of optimality}
In order to provide a sufficient condition of optimality, we need to introduce the
auxiliary Hamiltonian $\tilde{\mathcal{H}}: [0,T]\times\mathbb{R}%
^{n}\times\mathbb{R}\times\mathbb{R}^{d}\times U\times\mathbb{R}^{n}\times
\mathbb{R}^{n\times d}\longmapsto\mathbb{R}$ by
\[
\tilde{\mathcal{H}}(t,x,y,z,u,p,q)=p^{\intercal}b(t,x,u)+\sum\limits_{i=1}%
^{d}\left(  q^{i}\right)  ^{\intercal}\sigma^{i}(t,x,u)+f(t,x,y,z,u).
\]

Now we present the following sufficient condition of optimality. The proof can be found in Section \ref{pf-Th-4-2}.
\begin{theorem}\label{Th-suffi-cond}
Suppose that Assumptions \ref{assum-2} and \ref{assum-5} hold. Let $\bar{u}(\cdot) \in \mathcal{U}[0,T]$,
$(\bar{X}(\cdot),\bar{Y}(\cdot),\bar{Z}(\cdot))$ be the corresponding state trajectories and
$\left( p(\cdot)  ,q(\cdot)  \right)$ be the unique solution to the equation (\ref{eq-p}) corresponding
to $(\bar{X}(\cdot),\bar{Y}(\cdot),\bar{Z}(\cdot),\bar{u}(\cdot))$. We further assume that (\ref{local-SMP}) holds for
$\left( \bar{X}(\cdot),\bar{Y}(\cdot),\bar{Z}(\cdot), \bar{u}(\cdot), p(\cdot) ,q(\cdot) \right)$.
If
\begin{equation}
\Phi(X_{T}^{u})-\Phi(\bar{X}_{T})\geq\left(  \Phi_{x}(\bar{X}_{T})\right)
^{\intercal}(X_{T}^{u}-\bar{X}_{T}),\text{ \ }\mathbb{P}%
-a.s.\label{suffi-Phi-cond}%
\end{equation}
for any $u(\cdot) \in \mathcal{U}[0,T]$
and
\begin{equation}%
\begin{array}
[c]{rl}
& \mathcal{\tilde{H}}(t,x,y,z,p_{t},q_{t},u)-\mathcal{\tilde{H}}(t,\bar
{x},\bar{y},\bar{z},p_{t},q_{t},\bar{u})\\
\geq & \mathcal{\tilde{H}}_{x}^{\intercal}(t,\bar{x},\bar{y},\bar{z}%
,p_{t},q_{t},\bar{u})(x-\bar{x})+\mathcal{\tilde{H}}_{y}(t,\bar{x},\bar
{y},\bar{z},p_{t},q_{t},\bar{u})(y-\bar{y})\\
& +\mathcal{\tilde{H}}_{z}^{\intercal}(t,\bar{x},\bar{y},\bar{z},p_{t}%
,q_{t},\bar{u})(z-\bar{z})+\mathcal{\tilde{H}}_{u}^{\intercal}(t,\bar{x}%
,\bar{y},\bar{z},p_{t},q_{t},\bar{u})(u-\bar{u})\\
& -\sum\limits_{i=1}^{d}\mathcal{\tilde{H}}_{z_{i}}(t,\bar{x},\bar{y},\bar
{z},p_{t},q_{t},\bar{u})p_{t}^{\intercal}\left[  \sigma^{i}(t,x,u)-\sigma
^{i}(t,\bar{x},\bar{u})\right.  \\
& -\left.  \sigma_{x}^{i}(t,\bar{x},\bar{u})(x-\bar{x})-\sigma_{u}^{i}%
(t,\bar{x},\bar{u})(u-\bar{u})\right]  ,\quad dt\otimes d\mathbb{P}-a.e.
\end{array}
\label{suffi-tildeH-cond}%
\end{equation}
for any $(x,y,z,u),(\bar{x},\bar{y},\bar{z},\bar{u}) \in \mathbb{R}^{n}\times\mathbb{R}\times\mathbb{R}^{d}\times U$,
then $\bar{u}(\cdot)$ is an optimal control for the control problem (\ref{obje-eq}).
\end{theorem}

\begin{remark}\label{rmr-suffi-cond}
If $\sigma$ is linear with respect to $\left( x,u \right)$, then (\ref{suffi-tildeH-cond}) is nothing but the convexity condition of $\mathcal{\tilde{H}}$ with respect to $\left( x,y,z,u \right)$ since the last term on the right-hand side of the inequality (\ref{suffi-tildeH-cond}) vanishes in this case.
\end{remark}

\section{Appendix: Proof of Theorem \ref{Th-suffi-cond} \label{pf-Th-4-2}}

\begin{proof}
Denote by
\[%
\begin{array}
[c]{rl}%
\eta(t)= & Y_{t}^{u}-\bar{Y}_{t}-p_{t}^{\intercal}(X_{t}^{u}-\bar{X}_{t});\\
\zeta^{i}(t)= & \left(  Z_{t}^{u}\right)  ^{i}-\bar{Z}_{t}^{i}-\left(
q_{t}^{i}\right)  ^{\intercal}(X_{t}^{u}-\bar{X}_{t})\\
& -p_{t}^{\intercal}\left[  \sigma^{i}(t,X_{t}^{u},u_{t})-\sigma^{i}(t,\bar
{X}_{t},\bar{u}_{t})\right]  \text{ \ for \ }i=1,2,\ldots,d.
\end{array}
\]%
By applying It\^{o}'s formula to $\eta$, we get
\begin{equation}%
\begin{array}
[c]{rl}%
\eta(t)= & \Phi(X_{T}^{u})-\Phi(\bar{X}_{T})-\left(  \Phi_{x}(\bar{X}%
_{T})\right)  ^{\intercal}(X_{T}^{u}-\bar{X}_{T})\\
& +%
{\displaystyle\int_{t}^{T}}
\left\{  \mathcal{\tilde{H}}_{y}(s,\bar{X}_{s},\bar{Y}_{s},\bar{Z}_{s},\bar
{u}_{s},p_{s},q_{s})\eta(s)+\sum\limits_{i=1}^{d}\mathcal{\tilde{H}}_{z_{i}%
}(s,\bar{X}_{s},\bar{Y}_{s},\bar{Z}_{s},\bar{u}_{s},p_{s},q_{s})\zeta
^{i}(s)\right.  \\
& +\left.  R_{1}(s)+R_{2}(s)\right\}  ds-\sum\limits_{i=1}^{d}%
{\displaystyle\int_{t}^{T}}
\zeta^{i}(s)dW_{s}^{i},
\end{array}
\label{th-4-2-eq-3}%
\end{equation}
where%
\[%
\begin{array}
[c]{rl}%
R_{1}(s)= & \mathcal{\tilde{H}}(s,X_{s}^{u},Y_{s}^{u},Z_{s}^{u},u_{s},p_{s}%
,q_{s})-\mathcal{\tilde{H}}(s,\bar{X}_{s},\bar{Y}_{s},\bar{Z}_{s},\bar{u}%
_{s},p_{s},q_{s})\\
& -\mathcal{\tilde{H}}_{x}^{\intercal}(s,\bar{X}_{s},\bar{Y}_{s},\bar{Z}%
_{s},\bar{u}_{s},p_{s},q_{s})(X_{s}^{u}-\bar{X}_{s})-\mathcal{\tilde{H}}%
_{y}(s,\bar{X}_{s},\bar{Y}_{s},\bar{Z}_{s},\bar{u}_{s},p_{s},q_{s})(Y_{s}%
^{u}-\bar{Y}_{s})\\
& -\sum\limits_{i=1}^{d}\mathcal{\tilde{H}}_{z_{i}}(s,\bar{X}_{s},\bar{Y}%
_{s},\bar{Z}_{s},\bar{u}_{s},p_{s},q_{s})\left[  \left(  Z_{s}^{u}\right)
^{i}-\bar{Z}_{s}^{i}\right]  \\
& -\mathcal{\tilde{H}}_{u}^{\intercal}(s,\bar{X}_{s},\bar{Y}_{s},\bar{Z}%
_{s},\bar{u}_{s},p_{s},q_{s})(u_{s}-\bar{u}_{s})\\
& +\sum\limits_{i=1}^{d}\mathcal{\tilde{H}}_{z_{i}}(s,\bar{X}_{s},\bar{Y}%
_{s},\bar{Z}_{s},\bar{u}_{s},p_{s},q_{s})p_{s}^{\intercal}\left[  \sigma
^{i}(s,X_{s}^{u},u_{s})-\sigma^{i}(s,\bar{X}_{s},\bar{u}_{s})\right.  \\
& -\left.  \sigma_{x}^{i}(s,\bar{X}_{s},\bar{u}_{s})(X_{s}^{u}-\bar{X}%
_{s})-\sigma_{u}^{i}(s,\bar{X}_{s},\bar{u}_{s})(u_{s}-\bar{u}_{s})\right]
\end{array}
\]
and
\[%
\begin{array}
[c]{rl}%
R_{2}(s)= & \left[  \sum\limits_{i=1}^{d}\mathcal{\tilde{H}}_{z_{i}}(s,\bar
{X}_{s},\bar{Y}_{s},\bar{Z}_{s},\bar{u}_{s},p_{s},q_{s})p_{s}^{\intercal
}\sigma_{u}^{i}(s,\bar{X}_{s},\bar{u}_{s})+\mathcal{\tilde{H}}_{u}^{\intercal
}(s,\bar{X}_{s},\bar{Y}_{s},\bar{Z}_{s},\bar{u}_{s},p_{s},q_{s})\right]
(u_{s}-\bar{u}_{s})\\
= & \left[  \sum\limits_{i=1}^{d}f_{z_{i}}(s,\bar{X}_{s},\bar{Y}_{s},\bar
{Z}_{s},\bar{u}_{s})p_{s}^{\intercal}\sigma_{u}^{i}(s,\bar{X}_{s},\bar{u}%
_{s})+f_{u}^{\intercal}(s,\bar{X}_{s},\bar{Y}_{s},\bar{Z}_{s},\bar{u}%
_{s})\right.  \\
& +\left.  p_{s}^{\intercal}b_{u}(s,\bar{X}_{s},\bar{u}_{s})+\sum
\limits_{i=1}^{d}\left(  q_{s}^{i}\right)  ^{\intercal}\sigma_{u}^{i}%
(s,\bar{X}_{s},\bar{u}_{s})\right]  (u_{s}-\bar{u}_{s}).
\end{array}
\]
By Assumption 2.2, we get
$\left\vert \mathcal{\tilde{H}}_{z}(t,\bar{X}_{t},\bar{Y}_{t},\bar{Z}%
_{t},p_{t},q_{t},\bar{u}_{t})\right\vert \leq C\left(  1+\left\vert \bar
{Z}_{t}\right\vert \right)  $. As $\bar{Z}\cdot W$ is a BMO-martingale,
$\mathcal{E}\left(  \sum\limits_{i=1}^{d}\int_{0}^{T}%
\mathcal{\tilde{H}}_{z_{i}}(t,\bar{X}_{t},\bar{Y}_{t},\bar{Z}_{t},p_{t}%
,q_{t},\bar{u}_{t})dW_{t}^{i}\right)$ is a uniformly integrable martingale
and satisfies reverse H\"{o}lder's inequality.
Then, by Girsanov's theorem, we can
define a new probability $\mathbb{Q}$ and a Brownian motion $\tilde{W}$ under $\mathbb{Q}$ by%
\[
d\mathbb{Q}:=\mathcal{E}\left(  \sum\limits_{i=1}^{d}\int_{0}^{T}%
\mathcal{\tilde{H}}_{z_{i}}(t,\bar{X}_{t},\bar{Y}_{t},\bar{Z}_{t},p_{t}%
,q_{t},\bar{u}_{t})dW_{t}^{i}\right)  d\mathbb{P};\text{ }\tilde{W}_{t}%
^{i}:=W_{t}^{i}-\int_{0}^{t}f_{z_{i}}^{u,v}(s)ds,\text{ }i=1,2,\ldots,d.
\]
Denote by $\mathbb{E}_{Q}\left[  \cdot\right]  $ the mathematical
expectation corresponding to $\mathbb{Q}$. Thus, (\ref{th-4-2-eq-3}) can be further rewritten as
\[
\begin{array}
[c]{rl}%
\eta(t)= & \Phi(X_{T}^{u})-\Phi(\bar{X}_{T})-\left(  \Phi_{x}(\bar{X}%
_{T})\right)  ^{\intercal}(X_{T}^{u}-\bar{X}_{T})\\
& +%
{\displaystyle\int_{t}^{T}}
\left\{  \mathcal{\tilde{H}}_{y}(s,\bar{X}_{s},\bar{Y}_{s},\bar{Z}_{s},\bar
{u}_{s},p_{s},q_{s})\eta(s)+R_{1}(s)+R_{2}(s)\right\}  ds-\sum\limits_{i=1}%
^{d}%
{\displaystyle\int_{t}^{T}}
\zeta^{i}(s)d\tilde{W}_{s}^{i}.
\end{array}
\]
Define $\Lambda_{t}=\exp\left\{  \int_{0}^{t}\mathcal{\tilde{H}}_{y}(s,\bar{X}%
_{s},\bar{Y}_{s},\bar{Z}_{s},\bar{u}_{s},p_{s},q_{s})ds\right\}  $ for $t\in\lbrack0,T]$.
Then, by applying It\^{o}'s formula to $\Lambda\eta$, we have
\begin{equation}%
\begin{array}
[c]{rl}%
\eta(t)= & \Lambda_{t}^{-1}\Lambda_{T}\left\{  \Phi(X_{T}^{u})-\Phi(\bar
{X}_{T})-\left(  \Phi_{x}(\bar{X}_{T})\right)  ^{\intercal}(X_{T}^{u}-\bar
{X}_{T})\right\}  \\
& +%
{\displaystyle\int_{t}^{T}}
\Lambda_{t}^{-1}\Lambda_{s}\left\{  R_{1}(s)+R_{2}(s)\right\}  ds-\sum
\limits_{i=1}^{d}%
{\displaystyle\int_{t}^{T}}
\Lambda_{t}^{-1}\Lambda_{s}\zeta^{i}(s)d\tilde{W}_{s}^{i}.
\end{array}
\label{th-4-2-eq-4}%
\end{equation}
Note that $\Lambda$ is bounded, and recall that $\frac{ d\mathbb{Q}}{ d\mathbb{P}}$ satisfies reverse H\"{o}lder's inequality, $\left\Vert
p\right\Vert _{\infty}<\infty$, $q(\cdot)\in\bigcap_{\beta>1}\mathcal{M}%
_{\mathcal{F}}^{2,\beta}([0,T];\mathbb{R}_{{}}^{n\times d})$ and $\bar{Z}\cdot
W\in BMO$. By the definition of $\zeta$, Assumption 2.2 (ii), H\"{o}lder's
inequality, reverse H\"{o}lder's inequality and the energy inequality for $\bar{Z}\cdot W$, we can
show that $\mathbb{E}_{Q}\left[  \left(  \int_{0}^{T}\left\vert \zeta
(t)\right\vert ^{2}dt\right)  ^{\frac{1}{2}}\right]  <\infty$. Hence the
stochastic integral in (\ref{th-4-2-eq-4}) is a true martingale under $\mathbb{Q}$.

On the one hand, by virtue of (\ref{suffi-Phi-cond}) and (\ref{suffi-tildeH-cond}), we have
\[
\eta(T)=\Phi(X_{T}^{u})-\Phi(\bar{X}_{T})-\left(  \Phi_{x}(\bar{X}%
_{T})\right)  ^{\intercal}(X_{T}^{u}-\bar{X}_{T})\geq0,\text{ \ }%
\mathbb{Q}-a.s.
\]
and $R_{1}(\cdot) \geq 0, \text{ } dt\otimes d\mathbb{Q}-a.e.$ since $\mathbb{Q}$ is equivalent to $\mathbb{P}$.
On the other hand, by the assumption that (\ref{local-SMP}) holds for
$\left( \bar{X}(\cdot),\bar{Y}(\cdot),\bar{Z}(\cdot), \bar{u}(\cdot), p(\cdot) ,q(\cdot) \right)$, we deduce that $R_{2}(\cdot) \geq 0, \text{ } dt\otimes d\mathbb{Q}-a.e.$.
Then, by taking mathematical expectation $\mathbb{E}_{Q} \left [ \cdot \right ]$
on both sides of (\ref{th-4-2-eq-4}) when $t=0$, we have
\[
J\left(  u(\cdot)\right)  -J\left(  \bar{u}(\cdot)\right)  =\eta
(0)=\mathbb{E}_{Q}\left[  \Lambda_{T}\eta(T)+\int_{0}^{T}\Lambda_{t}\left\{
R_{1}\mathcal{(}t\mathcal{)}+R_{2}\mathcal{(}t\mathcal{)}\right\}  dt\right]
\geq0.
\]
As $u(\cdot) \in \mathcal{U}[0,T]$ is arbitrary, the proof is complete.
\end{proof}

\end{document}